\documentclass[10pt,twoside]{amsart}
\usepackage{geometry}
\usepackage[english]{babel}
\usepackage{graphicx}
\usepackage{amsmath}
\usepackage{amsfonts}

\numberwithin{equation}{section}
\newtheorem{teo}{Theorem}
\newtheorem{lemma}{Lemma}
\newtheorem{defi}{Definition}
\newtheorem{coro}{Corollary}
\newtheorem{prop}{Proposition}
\numberwithin{lemma}{section}
\numberwithin{prop}{section}
\numberwithin{teo}{section}
\numberwithin{defi}{section}
\numberwithin{coro}{section}
\numberwithin{figure}{section}

\begin{document}
\title{Curvature Motion in a Minkowski Plane}
\author{Vitor Balestro}
\address[V.Balestro and R.Teixeira]{ Instituto de Matem\'{a}tica -- UFF --
Niter\'{o}i -- Brazil}
\email{vitorbalestro@id.uff.br\\
ralph@mat.uff.br}
\author{Marcos Craizer}
\address[M.Craizer]{ Departamento de Matem\'{a}tica -- PUC-Rio -- Rio de
Janeiro -- Brazil}
\email{craizer@puc-rio.br}
\author{Ralph C. Teixeira}

\begin{abstract}
In this paper we study the curvature flow of a curve in a plane endowed with
a minkowskian norm whose unit ball is smooth. We show that many of the
properties known in the euclidean case can be extended (with due
adaptations) to this new situation. In particular, we show that simple,
closed, strictly convex, smooth curves remain so until the area enclosed by
them vanishes.\ Moreover, their isoperimetric ratios converge to the minimum
possible value, only attained by the minkowskian circle -- so these curves
converge to a minkowskian "circular point" as the enclosed area approaches
zero.
\end{abstract}

\maketitle

\section{\label{SecIntro}Introduction}

\noindent\textit{Note: Since we wrote this paper, it came to our attention that its main
results have already been published. See: }\\

\noindent[1] Gage, \textit{M., Evolving plane curves by curvature in relative geometries}, Duke Math J. 72 (1993) 441-466. \\
\noindent[2] Gage, M. \&  Li, Y., \textit{Evolving plane curves by curvature in relative geometries II}, Duke Math J. 75 (1994) 79-98. \\

Possibly the most fascinating front deformation, the classical planar
Curvature Motion is defined by%
\begin{equation*}
\frac{\partial \gamma }{\partial t}\left( u,t\right) =k\left( u,t\right)
N\left( u,t\right)
\end{equation*}%
where $k$ and $N$ are the curvature and the inwards unit normal vector to
the closed curve $\gamma \left( \cdot ,t\right) $ at the point $\gamma
\left( u,t\right) $. A series of papers (\cite{gage1}, \cite{gage2}, \cite%
{gage3} and \cite{grayson}) has shown that any embedded curve in the
Euclidean plane remains embedded and converges to a "circular point" in
finite time.\ Moreover, if $L\left( t\right) $ and $A\left( t\right) $ are
the length of $\gamma $ and the area it encloses, some very simple formulae
can be shown about their evolutions:%
\begin{align*}
\frac{dL}{dt}& =-\int k^{2}ds \\
\frac{dA}{dt}& =-2\pi \\
\lim_{t\rightarrow t_{V}}\frac{L^{2}}{A}& =4\pi
\end{align*}%
(where $t_{V}=\frac{A\left( 0\right) }{2\pi }$) the last one being one
interpretation of "converging to a circular shape".

On the other hand, a Minkowski plane is a 2-dimensional vector space with a
norm which can be defined by its unit ball $\mathcal{P}$ (a convex symmetric
set). Of course, along with a different geometry, come different notions of
lengths, normal vectors and curvature, which we very briefly review in the
next section (see \cite{thompson} for details and \cite{martini1}, \cite%
{martini2} for a survey). So it is natural to ask: are the properties of
Curvature Motion still valid on the Minkowski plane, with the due
adaptations? The goal of this paper is to answer a resounding YES, at least
when the boundary of $\mathcal{P}$ is smooth and the initial curve $\gamma
\left( \cdot ,0\right) $ is smooth and strictly convex. More specifically,
following similar techniques as in \cite{gage1}, \cite{gage2} and \cite%
{gage3}, we show that the flow is well defined up to the vanishing time $%
t_{V}=\frac{A\left( \gamma \left( 0\right) \right) }{2A\left( \mathcal{P}%
\right) }$, and that%
\begin{eqnarray*}
\frac{dL_{\mathcal{Q}}}{dt} &=&-\int k^{2}ds \\
\frac{dA}{dt} &=&-2A\left( \mathcal{P}\right) \\
\lim_{t\rightarrow t_{V}}\frac{L_{\mathcal{Q}}^{2}}{A} &=&4A\left( \mathcal{P%
}\right)
\end{eqnarray*}%
where $A\left( \mathcal{P}\right) $ is the area of the unit ball $\mathcal{P}
$ and all lengths are taken with respect to the metric defined by the dual
unit ball $\mathcal{Q}$.

The structure of this paper is as follows: section \ref{SecMink} briefly
reminds us of the basic ideas of Minkowski plane geometry, including some
notation choices. Section \ref{SecIso} states many interesting and necessary
minkowskian isoperimetric inequalities; it is divided in two subsections,
the first devoted to the minkowskian version of Gage%
\'{}%
s inequality and the second to a lemma with a more technical proof. Section %
\ref{SecCurv} defines the Minkowskian curvature flow and calculates the
evolutions of curvatures, lengths and areas as long as the flow is well
defined. Section \ref{SecConv} shows the convergence of the isoperimetric
ratio to the "circular" value $4A\left( \mathcal{P}\right) $ if the enclosed
area goes to $0$ and the curves remain simple and convex along the motion.
Finally, the technical section \ref{SecExist} has the job of showing the
existence of such a flow, all the way until the enclosed area converges to $%
0 $, at least when the initial curve is strictly convex and smooth, rounding
up the former results.

\textbf{Acknowledgment.} The authors would like to thank CNPq for financial
support during the preparation of this manuscript.

\section{\label{SecMink}Minkowski plane and its dual}

Let $\mathcal{P}$ be a strictly convex set, symmetric (which, throughout the
paper, will mean "symmetric with respect to the origin"), whose boundary is
given by a $C^{\infty }$ curve $p$. We endow the plane $\mathbb{R}^{2}$ with
a norm which makes $\mathcal{P}$ the unit ball. In other words, given $v\in 
\mathbb{R}^{2}$, write $v=tp$ for some $t\geq 0$ and some $p$ in the
boundary of $\mathcal{P}$, and define $||v||_{\mathcal{P}}=t$.

Denoting $e_{r}=(\cos \theta ,\sin \theta )$ and $e_{\theta }=(-\sin \theta
,\cos \theta )$ we parameterize $p$ by $p(\theta )$, $0\leq \theta \leq 2\pi 
$, such that $p^{\prime }(\theta )$ is a non-negative multiple of $e_{\theta
}$, i.e., the angle between the $x$-axis and $p^{\prime }\left( \theta
\right) $ is $\theta +\pi /2$. We can write%
\begin{align}
p(\theta )& =a(\theta )e_{r}+a^{\prime }(\theta )e_{\theta }
\label{eqn:Mink} \\
p^{\prime }\left( \theta \right) & =\left( a\left( \theta \right) +a^{\prime
\prime }\left( \theta \right) \right) e_{\theta }  \notag \\
\left[ p,p^{\prime }\right] & =a\left( \theta \right) \left( a\left( \theta
\right) +a^{\prime \prime }\left( \theta \right) \right)  \notag
\end{align}%
where $a(\theta )$ is the support function of $\mathcal{P}$. Furthermore, we
shall assume $a(\theta )+a^{\prime \prime }(\theta )>0$ for each $0\leq
\theta \leq 2\pi $, which is equivalent to say that the curvature of $p$ is
strictly positive.

The dual unity ball $\mathcal{P}^{\ast }$ can be naturally identified with a
convex set $\mathcal{Q}$ in the plane with $p^{\ast }(w)=[w,q]$ for any $%
w\in \mathbb{R}^{2}$. One can see that 
\begin{eqnarray}
q(\theta ) &=&\frac{p^{\prime }(\theta )}{[p(\theta ),p^{\prime }(\theta )]}=%
\frac{1}{a\left( \theta \right) }e_{\theta }  \label{eqn:2} \\
q^{\prime }\left( \theta \right) &=&-\frac{1}{a\left( \theta \right) }e_{r}-%
\frac{a^{\prime }\left( \theta \right) }{a^{2}\left( \theta \right) }%
e_{\theta }  \notag \\
\left[ q,q^{\prime }\right] &=&a^{-2}\left( \theta \right)  \notag
\end{eqnarray}%
is a parameterization of the boundary of $\mathcal{Q}$. It is not difficult
to see that $q$ is a convex symmetric curve with strictly positive curvature
as well. It also holds that 
\begin{equation}
p(\theta )=-\frac{q^{\prime }(\theta )}{[q(\theta ),q^{\prime }(\theta )]}%
=-a^{2}q^{\prime }\left( \theta \right)  \label{eqn:3}
\end{equation}

Given another closed, strictly convex curve $\gamma $, we can parameterize
it by $\theta $ such that $\gamma ^{\prime }(\theta )=\lambda (\theta
)q(\theta )$ (in fact, anytime we use the notation $f^{\prime }$ we mean
derivative with respect to this parameter $\theta $). The Minkowski $%
\mathcal{Q}$-length $L_{\mathcal{Q}}$ of $\gamma $ is defined as 
\begin{equation*}
L_{\mathcal{Q}}(\gamma )=\int_{0}^{2\pi }\lambda (\theta )\ d\theta
\end{equation*}%
which inspires another useful parameterization of $\gamma $ by its $\mathcal{%
Q}$-arclength parameter $s$:%
\begin{equation*}
s\left( \theta \right) =\int_{0}^{\theta }\lambda \left( \sigma \right)
d\sigma
\end{equation*}%
Sometimes we will need a third different parameterization $\gamma \left(
u\right) $ for such a curve. In that case, we define $v=\frac{ds}{du}$ so we
can write%
\begin{equation*}
ds=vdu=\lambda d\theta
\end{equation*}

If a $\mathcal{P}$-circle is tangent to $\gamma $ at $\gamma \left( \theta
\right) $, the line joining its center to $\gamma \left( \theta \right) $
must be parallel to $p\left( \theta \right) $. Thus, it is natural to define
the minkowskian unit normal to the curve $\gamma $ at the point $\gamma
(\theta )$ as $p(\theta )$. The inverse of the radius of a $\mathcal{P}$%
-circle which has a $3$-point contact with $\gamma $ at $\gamma \left(
\theta \right) $ is the \emph{minkowskian curvature}%
\begin{equation}
k(\theta )=\left[ p,p^{\prime }\right] \frac{d\theta }{ds}=\frac{%
[p,p^{\prime }]}{\lambda (\theta )}.  \label{eqn:5}
\end{equation}%
Other notions of minkowskian curvature are possible -- in \cite{petty} $%
k\left( \theta \right) $ is called "circular curvature" (see also \cite%
{craizer} and \cite{tabachnikov}).

Define the support function $f:[0,2\pi ]\rightarrow \mathbb{R}$ of $\gamma $
by $f(\theta )=[\gamma (\theta ),q(\theta )]$. Notice that we can take $f$
naturally on the parameter $s$. We have

\begin{prop}
\label{prop1} The following equalities hold:

\begin{description}
\item[(a)] $\displaystyle\int_0^{L_{\mathcal{Q}}}k(s)ds = 2A(\mathcal{P})$;

\item[(b)] $\displaystyle\int_{0}^{L_{\mathcal{Q}}}f(s)ds=2A$; and

\item[(c)] $\displaystyle\int_{0}^{L_{\mathcal{Q}}}f(s)k(s)ds=L_{\mathcal{Q}%
} $
\end{description}
\end{prop}

\begin{proof}
Since $ds=\lambda d\theta $, equation (\ref{eqn:5}) yields%
\begin{equation*}
\int_{0}^{L_{\mathcal{Q}}}k(s)ds=\int_{0}^{2\pi }k(\theta )\lambda (\theta
)d\theta =\int_{0}^{2\pi }[p(\theta ),p^{\prime }(\theta )]d\theta =2A(%
\mathcal{P})
\end{equation*}%
and this proves \textbf{(a)}. For \textbf{(b)} we calculate 
\begin{equation*}
\int_{0}^{L_{\mathcal{Q}}}f(s)ds=\int_{0}^{2\pi }[\gamma (\theta ),q(\theta
)]\lambda (\theta )d\theta =\int_{0}^{2\pi }[\gamma (\theta ),\gamma
^{\prime }(\theta )]d\theta =2A
\end{equation*}%
Now, for \textbf{(c)}, 
\begin{align*}
\int_{0}^{L_{\mathcal{Q}}}f(s)k(s)ds& =\int_{0}^{2\pi }[\gamma (\theta
),q(\theta )]k(\theta )\lambda (\theta )d\theta =\int_{0}^{2\pi }[\gamma
(\theta ),q(\theta )][p(\theta ),p^{\prime }(\theta )]d\theta = \\
& =\int_{0}^{2\pi }[\gamma (\theta ),p^{\prime }(\theta )]d\theta
=\int_{0}^{2\pi }[p(\theta ),\gamma ^{\prime }(\theta )]d\theta
=\int_{0}^{2\pi }\lambda (\theta )d\theta =L_{\mathcal{Q}}
\end{align*}
\end{proof}

\section{\label{SecIso}Some Isoperimetric Inequalities}

Consider again a smooth, closed and convex curve $\gamma $ with $\mathcal{Q}$%
-length $L_{\mathcal{Q}}$ enclosing the (usual) area $A$. The following
isoperimetric inequality generalizes the classical euclidean one (see Cap.4
of\textbf{\ \cite{thompson}}): 
\begin{equation}
\frac{L_{\mathcal{Q}}^{2}}{A}\geq 4A(\mathcal{P}),  \label{eqn:isop}
\end{equation}%
where $A(\mathcal{P})$ is the usual area of the unit $\mathcal{P}$-ball. As
in the euclidean case, the equality holds if and only if the curve is the
boundary of some $\mathcal{P}$-ball.

\subsection{The Minkowskian Gage Inequality}

We now turn our attention to prove a version of the Gage's inequality (see 
\textbf{\cite{gage1}}) in the Minkowski plane. Let $\mathcal{C}$ be the
space of smooth, simple, closed and strictly convex curves in the plane
endowed with the Hausdorff topology. We have:

\begin{teo}
\label{teoGage}There exists a non-negative, continuous, scale-invariant
functional $F:\mathcal{C}\rightarrow \mathbb{R}$ such that%
\begin{equation*}
\left( 1-F(\gamma )\right) \int_{0}^{L_{\mathcal{Q}}}k^{2}ds-A(\mathcal{P})%
\frac{L_{\mathcal{Q}}}{A}\geq 0,
\end{equation*}%
where $A$, $L_{\mathcal{Q}}$ and $k$ are the area, $\mathcal{Q}$-length and
curvature of $\gamma $. Moreover $F\left( \gamma \right) =0$ if and only if $%
\gamma $ is a $\mathcal{P}$-circle.
\end{teo}

\begin{coro}
\label{teo1} Given $\gamma \in \mathcal{C}$, we have%
\begin{equation}
\int_{0}^{L_{\mathcal{Q}}}k^{2}ds-A(\mathcal{P})\frac{L_{\mathcal{Q}}}{A}%
\geq 0.  \label{eqn:12}
\end{equation}%
with equality if and only if $\gamma $ is a $\mathcal{P}$-circle.
\end{coro}

In order to prove this, we need many results. We start by recalling an
useful Bonnesen inequality whose proof can be found in Theorem 4.5.5 of \cite%
{thompson}:

\begin{teo}
\label{lemma3} For $\gamma \in \mathcal{C}$, let $r_{\mathrm{in}}$ be the
radius of the biggest inscribed $\mathcal{P}$-circle and $r_{\mathrm{out}}$
the radius of the smallest circumscribed $\mathcal{P}$-circle. Then 
\begin{equation}
rL_{\mathcal{Q}}-A-A(\mathcal{P})r^{2}\geq 0  \label{eqn:14}
\end{equation}%
whenever $r_{\mathrm{in}}\leq r\leq r_{\mathrm{out}}$.
\end{teo}

\begin{lemma}
\label{teo01} The equality in (\ref{eqn:14}) holds for $r=r_{\mathrm{in}}$
if and only if $\gamma $ is homothetic to the $\mathcal{P}$-circle.
\end{lemma}

We have not seen a proof of this lemma in the literature, so we prove it in
the next subsection. Now let us begin to build the functional $F\left(
\gamma \right) $ of Gage%
\'{}%
s inequality:

\begin{prop}
\label{teo2} Consider the space $\mathcal{C}_{s}$ consisting of curves in $%
\mathcal{C}$ which are symmetric. Define the functional $E:\mathcal{C}%
_{s}\rightarrow \mathbb{R}$ by 
\begin{equation*}
E(\gamma )=1+\frac{A(\mathcal{P})r_{\mathrm{in}}r_{\mathrm{out}}}{A}-\frac{%
2A(\mathcal{P})(r_{\mathrm{in}}+r_{\mathrm{out}})}{L_{\mathcal{Q}}}
\end{equation*}%
Then, the following hold:\newline
\textbf{(1)} $L_{\mathcal{Q}}A\left( 1-E(\gamma )\right) \geq A(\mathcal{P})%
\displaystyle\int_{0}^{L_{\mathcal{Q}}}f^{2}ds$, for all $\gamma \in 
\mathcal{C}_{s}$;\newline
\textbf{(2)} $E(\gamma )\geq 0$ and equality holds if and only if $\gamma $
is a $\mathcal{P}$-circle; and\newline
\textbf{(3)} If $\gamma _{j}$ is a sequence in $\mathcal{C}_{s}$ such that $%
\displaystyle\lim_{j\rightarrow \infty }E(\gamma _{j})=0$ and if the
sequence of the normalized curves $\eta _{j}=\gamma _{j}\displaystyle\sqrt{%
\frac{A(\mathcal{P})}{A}}$ is contained in some bounded region of the plane,
then the region $H_{j}$ enclosed by $\eta _{j}$ converges in the Hausdorff
metric to $\mathcal{P}$, as $j\rightarrow \infty $.
\end{prop}

\begin{proof}
If $r_{\mathrm{in}}\leq r\leq r_{\mathrm{out}}$ then%
\begin{equation*}
\left( r-r_{\mathrm{in}}\right) \left( r_{\mathrm{out}}-r\right) \geq
0\Rightarrow r\left( r_{\mathrm{in}}+r_{\mathrm{out}}\right) -r_{\mathrm{in}%
}r_{\mathrm{out}}\geq r^{2}
\end{equation*}%
For a curve $\gamma $ in $\mathcal{C}_{s}$ the support function $f$
satisfies $r_{\mathrm{in}}\leq f\leq r_{\mathrm{out}}$ for every value of
the parameter, so we can take $r=f$ above. Then we integrate the above
inequality to obtain 
\begin{equation*}
\left( r_{\mathrm{out}}+r_{\mathrm{in}}\right) \int_{0}^{L_{\mathcal{Q}}}f\
ds-r_{\mathrm{in}}r_{\mathrm{out}}L_{\mathcal{Q}}\geq \int_{0}^{L_{\mathcal{Q%
}}}f^{2}ds
\end{equation*}%
Since $\displaystyle\int_{0}^{L_{\mathcal{Q}}}f\ ds=2A$ we have the
inequality in \textbf{\textit{(1)}}. For \textbf{\textit{(2)}} let $g\left(
r\right) =rL_{\mathcal{Q}}-A-A(\mathcal{P})r^{2}.$ From Theorem \ref{lemma3}%
.we know that $g\left( r_{\mathrm{in}}\right) ,g\left( r_{\mathrm{out}%
}\right) \geq 0,$ so we may write for $r\in \left[ r_{\mathrm{in}},r_{%
\mathrm{out}}\right] $%
\begin{equation*}
\frac{r-r_{\mathrm{in}}}{r_{\mathrm{out}}-r_{\mathrm{in}}}g\left( r_{\mathrm{%
out}}\right) +\frac{r_{\mathrm{out}}-r}{r_{\mathrm{out}}-r_{\mathrm{in}}}%
g\left( r_{\mathrm{in}}\right) \geq 0
\end{equation*}%
which can be rewritten as%
\begin{equation*}
r\left( L_{\mathcal{Q}}-A\left( \mathcal{P}\right) \left( r_{\mathrm{out}%
}+r_{\mathrm{in}}\right) \right) -A+A\left( \mathcal{P}\right) r_{\mathrm{in}%
}r_{\mathrm{out}}\geq 0.
\end{equation*}%
Taking $r=f$ and integrating with respect to $s$%
\begin{equation*}
2A\left( L_{Q}-A\left( P\right) \left( r_{\mathrm{out}}+r_{\mathrm{in}%
}\right) \right) -AL_{\mathcal{Q}}+A\left( \mathcal{P}\right) r_{\mathrm{in}%
}r_{\mathrm{out}}L_{\mathcal{Q}}\geq 0
\end{equation*}%
which shows that $E\left( \gamma \right) \geq 0$. Since $f$ ranges from $r_{%
\mathrm{in}}$ to $r_{\mathrm{out}}$, if $E(\gamma )=0$ then we must have $%
g\left( r_{\mathrm{in}}\right) =g\left( r_{\mathrm{out}}\right) =0$ and
Lemma \ref{teo01} says that $\gamma $ is a $\mathcal{P}-$circle.

For \textit{\textbf{(3)}} let $\gamma _{j}$ be a sequence in $\mathcal{C}%
_{s} $ such that $\displaystyle\lim_{j\rightarrow \infty }E(\gamma _{j})=0$
and assume that all normalized curves $\eta _{j}$ lie at a same bounded
region of the plane. Notice that $E(\eta _{j})=E(\gamma _{j})$ for every $%
j\in \mathbb{N}$ and then $\displaystyle\lim_{j\rightarrow \infty }E(\eta
_{j})=0$. Denote by $H_{j}$ the region enclosed by $\eta _{j}$. By
Blaschke's Selection Theorem we have that there exists a subsequence $%
H_{j_{k}}$ which converges to a convex set $H$. Since $E$ is a continuous
functional (considering the Hausdorff topology in $\mathcal{C}_{s}$) we have 
$E(H)=\displaystyle\lim_{k\rightarrow \infty }E(H_{j_{k}})=0$, and then $H$
must be the unit $\mathcal{P}$-circle. It is also true that every convergent
subsequence of $H_{j}$ converges to the unit $\mathcal{P}$-circle. It
follows immediately that $H_{j}$ itself converges to the unit $\mathcal{P}$%
-circle. This concludes the proof.
\end{proof}

Finally we appropriately extend the functional $E$ to the desired functional 
$F$, as done in \cite{gage2}. Let $\gamma \in \mathcal{C}$, and consider all
chords which divide the area inside $\gamma $ in two equal parts. Pick one
(call it $S$) such that the tangent lines to $\gamma $ on the extreme points
are parallel. Let $\gamma _{1}$ (with $\mathcal{Q}$-length $L_{1}$) and $%
\gamma _{2}$ (with $\mathcal{Q}$-length $L_{2}$) be the two portions of $%
\gamma $ determined by $S$. Placing the $x$-axis along $S$ and the origin at
its midpoint we can build two curves $\gamma _{1}^{\ast }$ and $\gamma
_{2}^{\ast }$ which belongs to $\mathcal{C}_{s}$ by reflecting $\gamma _{1}$
and $\gamma _{2}$ through the origin. Since the functional $E$ is well
defined for these new curves we could define $F(\gamma )$ by 
\begin{equation*}
F(\gamma )=\frac{L_{1}}{L_{\mathcal{Q}}}E(\gamma _{1}^{\ast })+\frac{L_{2}}{%
L_{\mathcal{Q}}}E(\gamma _{2}^{\ast }),
\end{equation*}%
but, although this definition looks natural (because it coincides with $E$
in $\mathcal{C}_{s}$), it is not always correct. This happens because the
choice of the chord $S$ is not necessarily unique. To overcome this trouble
we define $F$ to be the supremum of the above expression between all
possible choices of $S$. It is not difficult to prove that the functional $F$
has also the properties \textbf{\textit{(1)}}, \textbf{\textit{(2)}} and 
\textbf{\textit{(3)}}, and we will omit the details.

Now we can finish the proof of Theorem \ref{teoGage}. By Schwarz inequality
we have 
\begin{equation*}
L_{\mathcal{Q}}=\int_{0}^{L_{\mathcal{Q}}}fk\ ds\leq \left( \int_{0}^{L_{%
\mathcal{Q}}}f^{2}\ ds\right) ^{1/2}\left( \int_{0}^{L_{\mathcal{Q}%
}}k^{2}ds\right) ^{1/2}
\end{equation*}%
Squaring both sides and using inequality expressed in Theorem \ref{teo2}
yields 
\begin{equation*}
L_{\mathcal{Q}}^{2}\leq \left( \int_{0}^{L_{\mathcal{Q}}}f^{2}\ ds\right)
\left( \int_{0}^{L_{\mathcal{Q}}}k^{2}ds\right) \leq \frac{L_{\mathcal{Q}}A}{%
A(\mathcal{P})}\left( 1-F(\gamma )\right) \int_{0}^{L_{\mathcal{Q}}}k^{2}ds
\end{equation*}%
and the desired result comes immediately.

\subsection{Proof of Lemma \protect\ref{teo01}}

Assume that $\gamma \in {\mathcal{C}}$ is symmetric. We start with the
following quite intuitive result:

\begin{lemma}
\label{lemma:rinmu0} Denote by $\mu _{0}$ the minimum curvature radius of $%
\gamma $. Then $\mu _{0}\leq r_{\mathrm{in}}$ with equality only in case $%
\gamma $ is a $\mathcal{P}$-circle.
\end{lemma}

\begin{proof}
Assume that $\mu (0)=\mu _{0}$, where $\mu _{0}=\min \{\mu (\theta )|0\leq
\theta \leq 2\pi \}$. We may also assume that $(-a,a)$, $0\leq a<\frac{\pi }{%
2}$, is the maximal interval where $\mu (\theta )=\mu _{0}$. Observe that if 
$a=\frac{\pi }{2}$ then, by the symmetry of $\gamma $, $\gamma $ would
necessarily be a $\mathcal{P}$-circle.

For any $-\pi \leq \theta \leq \pi $, denote $P_{0}=(x_{0}(\theta
),y_{0}(\theta ))$ the $\mathcal{P}$-circle of radius $\mu _{0}$ osculating
at $\gamma (0)$. In the euclidean case, $x_{0}=\mu _{0}\sin (\theta )$, $%
y_{0}=\mu _{0}-\mu _{0}\cos (\theta )$. Denote also $\gamma (\theta
)=(x(\theta ),y(\theta ))$, $q(\theta )=(q_{1}(\theta ),q_{2}(\theta ))$ and 
$p(\theta )=(p_{1}(\theta ),p_{2}(\theta ))$. Since 
\begin{equation*}
\gamma ^{\prime }(\theta )=\mu (\theta )[p,p^{\prime }](\theta )q(\theta ),
\end{equation*}%
we can write 
\begin{equation*}
y(\theta )=\int_{0}^{\theta }\mu \lbrack p,p^{\prime }]q_{2}d\theta \geq \mu
_{0}\left( p_{2}(\theta )-p_{2}(0)\right) =y_{0}(\theta ),
\end{equation*}%
with equality if and only if $-a\leq \theta \leq a$. We conclude that $%
y(\theta )\geq y_{0}(\theta )$, and the equality holds if and only if $%
-a\leq \theta \leq a$. Thus the osculating $\mathcal{P}$-circle $P_{0}$ is
tangent to $\gamma $ only at $\gamma (\theta )$, $-a\leq \theta \leq a$. So
there exists $\epsilon >0$ such that $P_{0}+(0,\epsilon )$ is contained in
the interior of $\gamma $, thus proving the proposition.
\end{proof}

For $0\leq r\leq r_{\mathrm{in}}$, denote $D_{r}$ the set of points inside $%
\gamma $ whose distance to $\gamma $ is $\geq r$ and let $C_{r}=\partial
D_{r}$. Denote by $L_{\mathcal{Q}}(r)$ the $\mathcal{Q}$-length of $C_{r}$.
The following proposition is easy to prove:

\begin{prop}
\label{prop:ALr} If $r\leq \mu _{0}$, then 
\begin{equation}
L_{\mathcal{Q}}(r)=L_{\mathcal{Q}}-2A(\mathcal{P})r.  \label{eq:Lr}
\end{equation}%
Moreover,%
\begin{equation}
A=\int_{0}^{r_{\mathrm{in}}}L_{\mathcal{Q}}(r)dr.  \label{eq:area}
\end{equation}
\end{prop}

\begin{proof}
Let%
\begin{equation*}
\beta (\theta ,r)=\gamma (\theta )-rp(\theta )
\end{equation*}%
be a parameterization of $C_{r}$, $\theta \in I(r)$. If $r<\mu _{0}$, then $%
I(r)=[0.2\pi ]$ and so 
\begin{equation*}
L_{\mathcal{Q}}(r)=\int_{0}^{2\pi }[p,p^{\prime }](\mu (\theta )-r)d\theta
=L_{\mathcal{Q}}(0)-2A(P)r,
\end{equation*}%
which proves equation (\ref{eq:Lr}). To prove equation (\ref{eq:area}),
observe that the region $D$ enclosed by $\gamma $ is the disjoint union of $%
C_{r}$, $0\leq r\leq r_{\mathrm{in}}$. Since 
\begin{equation*}
\left[ \frac{\partial \beta }{\partial \theta },\frac{\partial \beta }{%
\partial r}\right] =[p,p^{\prime }]\left( \mu (\theta )-r\right) .
\end{equation*}%
and $\mu (\theta )-r>0$, for $\theta \in I(r)$, we conclude that 
\begin{equation*}
A=\int_{r=0}^{r_{in}}\int_{I(r)}[p,p^{\prime }]\left( \mu (\theta )-r\right)
d\theta dr=\int_{r=0}^{r_{in}}L_{\mathcal{Q}}(r)dr.
\end{equation*}
\end{proof}

Now consider an arc of the unit $\mathcal{P}$-circle defined by $\theta
_{1}\leq \theta \leq \theta _{2}$. Taking the tangents to $\mathcal{P}$ at $%
\theta =\theta _{1}$ and $\theta =\theta _{2}$ we obtain a polygonal line
formed by a pair of segments (see Figure \ref{F1}). It is not difficult to
verify that the $Q$-lengths of the segments are 
\begin{equation}
L_{1}=\frac{1-[p(\theta _{1}),q(\theta _{2})]}{[q(\theta _{1}),q(\theta
_{2})]};\ \ L_{2}=\frac{1-[p(\theta _{2}),q(\theta _{1})]}{[q(\theta
_{1}),q(\theta _{2})]}.
\end{equation}%
The $Q$-length of the arc is given by 
\begin{equation}
L_{arc}=\int_{\theta _{1}}^{\theta _{2}}[p,p^{\prime }]d\theta ,
\end{equation}%
and we define 
\begin{equation}
\delta (\theta _{1},\theta _{2})=L_{1}+L_{2}-L_{arc}.
\end{equation}%
Then $\delta =\delta (\theta _{1},\theta _{2})$ is strictly positive (see 
\cite{thompson}).

\begin{figure} 
\centering
\includegraphics{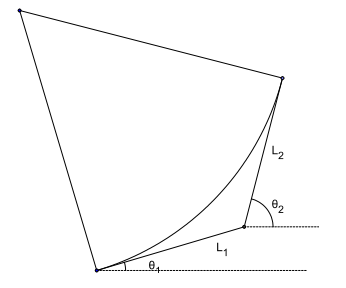}
\caption{Tangent segments to $\mathcal{P}$ at $\theta _{1}$ and $\theta _{2}$.}{\label{F1}}
\end{figure}

\begin{figure} {\label{F2}}
\centering
\includegraphics{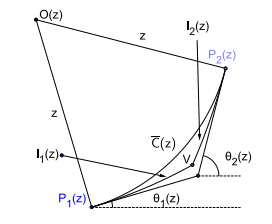}
\caption{Arcs of $C$ between $\theta _{1}$ and $\theta _{2}$ and $V$.}{\label{F2}}
\end{figure}

For $r>\mu _{0}$, the curves $C_{r}$ necessarily admit corners. Thus we must
consider curves which are smooth by parts with a finite number of vertices.

\begin{lemma}
\label{lemma:corner} Assume that $C$ is smooth by parts and at some corner $%
V $ the parameter of the tangent lines are $\theta _{1}$ and $\theta _{2}$.
Consider an arc of $\mathcal{P}$-circle of radius $z$ inscribed in this
corner. Denote by $l_{1}(z)$ and $l_{2}(z)$ the $Q$-lengths of the arcs of $%
C $ between the vertex $V$ and the tangency points $P_{1}\left( z\right) $
and $P_{2}\left( z\right) $(see Figure \ref{F2}). Then 
\begin{equation*}
l_{1}(z)+l_{2}(z)=zL_{arc}+z\delta \left( \theta _{1},\theta _{2}\right)
+O(z^{2}),
\end{equation*}%
where $\lim_{z\rightarrow 0}\frac{O(z^{2})}{z}=0.$
\end{lemma}

\begin{proof}
We parameterize $C$ around $V$ by an $\mathcal{Q}$-arclength parameter $s$
using a function $g:\left( -\varepsilon ,\varepsilon \right) \rightarrow 
\mathbb{R}^{2}$ with $g\left( 0\right) =V$ (so $g$ is smooth everywhere
except at $s=0$, where we have lateral derivatives%
\begin{equation*}
\frac{dg}{ds}\left( 0_{-}\right) =q\left( \theta _{1}\right) \ \mathrm{and} \ %
\frac{dg}{ds}\left( 0_{+}\right) =q\left( \theta _{2}\right)
\end{equation*}%
The definition of the tangent points means that $P_{1}\left( z\right)
=g\left( -l_{1}\left( z\right) \right) $ and $P_{2}\left( z\right) =g\left(
l_{2}\left( z\right) \right) $. Writing the position of the center $O\left(
z\right) $ in two ways, we write $l_{1}$ and $l_{2}$ implicitly as a
function of $z$:%
\begin{equation*}
g\left( -l_{1}\left( z\right) \right) -zp\left( \theta _{1}\left( z\right)
\right) =g\left( l_{2}\left( z\right) \right) -zp\left( \theta _{2}\left(
z\right) \right)
\end{equation*}%
where $\theta _{1}\left( z\right) $ and $\theta _{2}\left( z\right) $ are
the $\theta $-parameters associated to $P_{1}$ and $P_{2}$. Take derivatives
with respect to $z$ and then take $z\rightarrow 0$ (so $\theta _{1}\left(
z\right) \rightarrow \theta _{1}$ and $\theta _{2}\left( z\right)
\rightarrow \theta _{2}$) to arrive at%
\begin{equation*}
-q\left( \theta _{1}\right) \frac{dl_{1}}{dz}\left( 0\right) -p\left( \theta
_{1}\right) =q\left( \theta _{2}\right) \frac{dl_{2}}{dz}\left( 0\right)
-p\left( \theta _{2}\right)
\end{equation*}%
Now, this equation depends only on the angles $\theta _{1}$ and $\theta _{2}$
-- the exact shape of $C$ does not matter at all! So, up to first order, the
lengths $l_{1}$ and $l_{2}$ depend on $z$ the same way they would if the
curve were already the polygonal in \ref{F1} (scaled by a factor of $z$),
that is%
\begin{eqnarray*}
l_{1}\left( z\right) +l_{2}\left( z\right) &=&z\left( L_{1}+L_{2}\right)
+O\left( z^{2}\right) = \\
&=&zL_{arc}+z\delta \left( \theta _{1},\theta _{2}\right) +O\left(
z^{2}\right) .
\end{eqnarray*}
\end{proof}

\begin{prop}
\label{prop:corner} Consider a convex curve $C$ with at least one corner.
Then 
\begin{equation*}
\frac{dL_{\mathcal{Q}}}{ds}(0)<-2A(\mathcal{P}).
\end{equation*}
\end{prop}

\begin{proof}
Denote $\theta _{1}<\theta _{2}$ the angles of a corner $k$ and consider an
arc of the circle $\mathcal{P}$ defined by the angles $\theta _{1}$ and $%
\theta _{2}$. Consider $z$ small and inscribe a circle $z\mathcal{P}$ at a
corner $k$. Denote $\overline{C}_{z}$ the curve obtained from $C$ by
substituting each corner by the corresponding arc of the circle $z\mathcal{P}
$. By Lemma \ref{lemma:corner}, the length difference at the corner $k$ is $%
l_{k}(C)-L_{k}(\overline{C}_{z})=z\delta _{k}+O(z^{2})$. Since 
\begin{equation*}
L_{\mathcal{Q}}(\overline{C}_{z})=L_{\mathcal{Q}}(z)+2A(\mathcal{P})z,
\end{equation*}%
we conclude that 
\begin{equation*}
L_{\mathcal{Q}}(0)=L_{\mathcal{Q}}(z)+\left( 2A(\mathcal{P})+\sum_{k}\delta
_{k}\right) z+O(z^{2}),
\end{equation*}%
thus proving the proposition.
\end{proof}

\begin{coro}
\label{cor:Lr} If $r>\mu _{0}$ then $L_{\mathcal{Q}}(r)<L_{\mathcal{Q}}-2A(%
\mathcal{P})r$.
\end{coro}

\begin{proof}
By Proposition \ref{prop:corner}, $\frac{dL_{Q}}{ds}\leq -2A(\mathcal{P})$
with equality if and only if $s\leq \mu _{0}$. In fact, $C_{r}$ has a corner
if and only if $r>\mu _{0}$. Integrating from $0$ to $r$ we obtain $L_{%
\mathcal{Q}}(r)<L_{\mathcal{Q}}-2A(\mathcal{P})r$.
\end{proof}

Now we can complete the proof of Lemma \ref{teo01}: if equality holds in \ref%
{eqn:14}, then Proposition \ref{prop:ALr} and Corollary \ref{cor:Lr} imply
that $r_{\mathrm{in}}\leq \mu _{0}$. But then Lemma \ref{lemma:rinmu0}
implies that $r_{\mathrm{in}}=\mu _{0}$ and $\gamma $ is a $P$-circle.

\textbf{Remark.} Lemma \ref{teo01} is not necessarily true if $\gamma $ and
the $\mathcal{P}$-ball are not smooth! A counterexample: take $\mathcal{P}$
to be the square whose vertices are $(\pm 1,\pm 1)$ (so $\mathcal{Q}$ will
be the square $\left\vert x\right\vert +\left\vert y\right\vert \leq 1$) and 
$\gamma $ to be the rectangle with vertices $(\pm 2,\pm 1)$.

\section{\label{SecCurv}The minkowskian curvature flow}

We define the minkowskian curvature flow to be a family of closed curves $%
F:S^{1}\times \lbrack 0,T)\rightarrow \mathbb{R}^{2}$ satisfying the
following:%
\begin{align}
\frac{\partial F}{\partial u}(u,t)& =v(u,t).q(\theta (u,t));\ \mathrm{and}
\label{eqn:7} \\
\frac{\partial F}{\partial t}(u,t)& =-k(u,t).p(\theta (u,t))  \notag \\
F\left( u,0\right) & =\gamma \left( u\right)  \notag
\end{align}%
where $\gamma $ is a simple closed curve and, as usual, $\theta (u,t)$ is
defined such that the angle between the $x$-axis and $\partial F/\partial u$
at the point $(u,t)$ is $\theta (u,t)+\pi /2$.

\begin{lemma}
\label{lemma1} The following hold at each point of the flow:

\begin{description}
\item[(a)] $\displaystyle\frac{\partial v}{\partial t}=-k^{2}v$; and

\item[(b)] $\displaystyle\frac{\partial \theta }{\partial t}=\frac{1}{v\left[
q,q^{\prime }\right] }\frac{\partial k}{\partial u}=\frac{1}{\left[
q,q^{\prime }\right] }\frac{\partial k}{\partial s}$
\end{description}
\end{lemma}

\begin{proof}
Notice first that 
\begin{equation*}
\frac{\partial }{\partial t}\left( \frac{\partial F}{\partial u}\right) =%
\frac{\partial v}{\partial t}.q+v\frac{\partial \theta }{\partial t}%
.q^{\prime }=\frac{\partial v}{\partial t}.q-v\frac{\partial \theta }{%
\partial t}[q,q^{\prime }].p
\end{equation*}%
Now, 
\begin{equation*}
\frac{\partial }{\partial u}\left( \frac{\partial F}{\partial t}\right) =-%
\frac{\partial k}{\partial u}.p-k\frac{\partial \theta }{\partial u}%
.p^{\prime }=-\frac{\partial k}{\partial u}.p-k\frac{v}{\lambda }%
[p,p^{\prime }].q=-\frac{\partial k}{\partial u}.p-k^{2}v.q
\end{equation*}%
Then the result follows since $p$ and $q$ are always linearly independent.
\end{proof}

\begin{lemma}
\label{lemma2} The evolutions of the $\mathcal{Q}$-arclength and of the area
are given respectively by 
\begin{equation*}
\frac{dL_{\mathcal{Q}}}{dt}=-\int_{0}^{L_{\mathcal{Q}}(t)}k^{2}(s,t)ds;\ 
\mathrm{and}
\end{equation*}%
\begin{equation*}
\frac{dA}{dt}=-2A(\mathcal{P})
\end{equation*}
\end{lemma}

\begin{proof}
Since $L_{\mathcal{Q}}(t)=\displaystyle\int_{0}^{2\pi }v(u,t)du$ we have 
\begin{equation*}
\frac{dL_{\mathcal{Q}}}{dt}=\int_{0}^{2\pi }\frac{\partial v}{\partial t}%
du=-\int_{0}^{2\pi }k^{2}v\ du=-\int_{0}^{L_{\mathcal{Q}}}k^{2}\ ds
\end{equation*}%
The area $A(t)$ of the curve at time $t$ is given by 
\begin{equation*}
\frac{dA}{dt}=\displaystyle\frac{1}{2}\int_{0}^{2\pi }\left[ F(u,t),\frac{%
\partial F}{\partial u}(u,t)\right] du
\end{equation*}%
\newline
Thus, integrating by parts,%
\begin{align*}
\frac{dA}{dt}& =\frac{1}{2}\int_{0}^{2\pi }\left[ \frac{\partial F}{\partial
t},\frac{\partial F}{\partial u}\right] du+\frac{1}{2}\int_{0}^{2\pi }\left[
F,\frac{\partial ^{2}F}{\partial t\partial u}\right] du= \\
& =\int_{0}^{2\pi }\left[ \frac{\partial F}{\partial t},\frac{\partial F}{%
\partial u}\right] du=\int_{0}^{2\pi }[-kp,vq]du= \\
& =\int_{0}^{2\pi }-kv\ du=-\int_{0}^{L_{\mathcal{Q}}}k\ ds=-2A(\mathcal{P})
\end{align*}
\end{proof}

With these evolution formulae one can easily show that the evolution of the
isoperimetric ratio is 
\begin{equation}
\frac{d}{dt}\left( \frac{L_{\mathcal{Q}}^{2}}{A}\right) =-\frac{2L_{\mathcal{%
Q}}}{A}\left( \int_{0}^{L_{\mathcal{Q}}}k^{2}ds-A(\mathcal{P})\frac{L_{%
\mathcal{Q}}}{A}\right)  \label{eqn:disodt}
\end{equation}%
which, given (\ref{eqn:12}), shows that the isoperimetric ratio is
nonincreasing along the motion. In the next section we will prove that, as
in the euclidean case, if the flow continues until the area converges to
zero and the curves remain simple and convex along the motion, then the
isoperimetric ratio converges to the optimum value $4A(\mathcal{P})$. But
first, we establish the evolution of the curvature function.

\begin{lemma}
\label{lemmapde} The minkowskian curvature $k$ evolves according to the PDE%
\begin{equation*}
\frac{\partial k}{\partial \tau }=\frac{a}{a+a^{\prime \prime }}k^{2}\frac{%
\partial ^{2}k}{\partial \theta ^{2}}+\frac{2a^{\prime }}{a+a^{\prime \prime
}}k^{2}\frac{\partial k}{\partial \theta }+k^{3},
\end{equation*}%
where $\tau $ is the time parameter which is independent with $\theta $.
\end{lemma}

\begin{proof}
Using Lemma \ref{lemma1}\textbf{(a)} and $ds=vdu$, we arrive at%
\begin{equation*}
\frac{\partial }{\partial t}\frac{\partial }{\partial s}-\frac{\partial }{%
\partial s}\frac{\partial }{\partial t}=k^{2}\frac{\partial }{\partial s}
\end{equation*}%
just as in the Euclidean case. We apply this to the function $\theta =\theta
(s,t)$ and use $ds=\lambda d\theta $ and Lemma \ref{lemma1}\textbf{(b)} to
obtain 
\begin{equation*}
\frac{\partial }{\partial t}\left( \frac{k}{[p,p^{\prime }]}\right) -\frac{%
\partial }{\partial s}\left( \frac{1}{\left[ q,q^{\prime }\right] }\frac{%
\partial k}{\partial s}\right) =\frac{k^{3}}{[p,p^{\prime }]}.
\end{equation*}%
Unfortunately $p$ and $q$ now depend on $t$ as well, so, using equations (%
\ref{eqn:Mink}) and (\ref{eqn:2}), we arrive at%
\begin{equation*}
\frac{1}{a\left( a+a^{\prime \prime }\right) }\frac{\partial k}{\partial t}-%
\frac{k}{a^{2}\left( a+a^{\prime \prime }\right) ^{2}}\left[ a^{\prime
}\left( a+a^{\prime \prime }\right) +a\left( a^{\prime }+a^{\prime \prime
\prime }\right) \right] \frac{\partial \theta }{\partial t}-2aa^{\prime }%
\frac{\partial k}{\partial s}\frac{\partial \theta }{\partial s}-a^{2}\frac{%
\partial ^{2}k}{\partial s^{2}}=\frac{k^{3}}{a\left( a+a^{\prime \prime
}\right) }
\end{equation*}%
Now we change all $s$-derivatives to $\theta $-derivatives using equation (%
\ref{eqn:5}), and use Lemma \ref{lemma1}\textbf{(b)} to eventually get
to%
\begin{equation*}
\frac{\partial k}{\partial t}-\frac{2a^{\prime }}{a+a^{\prime \prime }}k^{2}%
\frac{\partial k}{\partial \theta }-\frac{a}{a+a^{\prime \prime }}k\left( 
\frac{\partial k}{\partial \theta }\right) ^{2}-\frac{a}{a+a^{\prime \prime }%
}k^{2}\frac{\partial ^{2}k}{\partial \theta ^{2}}=k^{3}
\end{equation*}%
Now, writing $k=k(\theta ,\tau )$ yields 
\begin{equation*}
\frac{\partial k}{\partial t}=\frac{\partial k}{\partial \theta }\frac{%
\partial \theta }{\partial t}+\frac{\partial k}{\partial \tau }.
\end{equation*}%
Using this (and replacing once again $\displaystyle\frac{\partial \theta }{%
\partial t}$ using Lemma \ref{lemma1}\textbf{(b))} we finish the proof.
\end{proof}

\section{\label{SecConv}Convergence of the isoperimetric ratio}

We now turn to show that the flow rounds the curves if they approach a
vanishing point. In the following $\gamma (u,t):S^{1}\times \lbrack
0,T)\rightarrow \mathbb{R}^{2}$ is a family (on parameter $t$) of curves in $%
\mathcal{C}$ which solves the minkowskian curvature flow (in the next
section, we will show that $\gamma \left( \cdot ,0\right) \in \mathcal{%
C\Rightarrow \gamma }\left( \cdot ,t\right) \in \mathcal{C}$). The $\mathcal{%
Q}$-length and the area of the curve at time $t$ are denoted, as usual, by $%
L_{\mathcal{Q}}(t)$ and $A(t)$.

\begin{lemma}
\label{lemma6} If $\displaystyle\lim_{t\rightarrow T}A(t)=0$ then 
\begin{equation*}
\liminf_{t\rightarrow T}L_{\mathcal{Q}}(t)\left( \int_{0}^{L_{\mathcal{Q}%
}(t)}k^{2}ds-A(\mathcal{P})\frac{L_{\mathcal{Q}}(t)}{A(t)}\right) =0
\end{equation*}
\end{lemma}

\begin{proof}
Suppose there exist $\epsilon >0$ and $t_{1}\in (0,T)$ such that 
\begin{equation*}
L_{\mathcal{Q}}(t)\left( \int_{0}^{L_{\mathcal{Q}}(t)}k^{2}ds-A(\mathcal{P})%
\frac{L_{\mathcal{Q}}(t)}{A(t)}\right) >\epsilon
\end{equation*}%
for every $t\in (t_{1},T)$. Put $g(t)=\log (A(t))$ for $t\in \lbrack 0,T)$.
Using the evolution of the isoperimetric ratio (\ref{eqn:disodt}) we have 
\begin{equation*}
\frac{d}{dt}\left( \frac{L_{\mathcal{Q}}^{2}}{A}\right) \leq -\frac{2}{A}%
\epsilon =\frac{\epsilon }{A(\mathcal{P})}\frac{dg}{dt}
\end{equation*}%
Fix $t\in (t_{1},T)$. Isoperimetric inequality (\ref{eqn:isop}) and
integration (from $t_{1}$ to $t$) yield 
\begin{equation*}
4A(\mathcal{P})\leq \frac{L_{\mathcal{Q}}^{2}(t)}{A(t)}\leq \frac{L_{%
\mathcal{Q}}^{2}(t_{1})}{A(t_{1})}-\frac{\epsilon }{A(\mathcal{P})}\log
(A(t_{1}))+\frac{\epsilon }{A(\mathcal{P})}\log (A(t))
\end{equation*}%
But the right hand side goes to $-\infty $ as $t$ converges to $T$. This
contradiction completes the proof.
\end{proof}

Now we are ready to prove the main theorem of this section.

\begin{teo}
\label{teo3} If $\displaystyle\lim_{t\rightarrow T}A(t) = 0$ then $%
\displaystyle\lim_{t \rightarrow T} \frac{L_{\mathcal{Q}}^2(t)}{A(t)} = 4A(%
\mathcal{P})$
\end{teo}

\begin{proof}
First we rewrite the inequality in Theorem \ref{teoGage} as 
\begin{equation*}
\int_{0}^{L_{\mathcal{Q}}}k^{2}ds-A(\mathcal{P})\frac{L_{\mathcal{Q}}}{A}%
\geq F(\gamma )\int_{0}^{L_{\mathcal{Q}}}k^{2}ds
\end{equation*}%
Schwarz inequality yields 
\begin{equation*}
L_{\mathcal{Q}}\int_{0}^{L_{\mathcal{Q}}}k^{2}ds\geq \left( \int_{0}^{L_{%
\mathcal{Q}}}k\ ds\right) ^{2}=4A(\mathcal{P})^{2}
\end{equation*}%
Combining both inequalities we have the following inequality for each curve $%
\gamma (\cdot ,t)$: 
\begin{equation*}
L_{\mathcal{Q}}\left( \int_{0}^{L_{\mathcal{Q}}}k^{2}ds-A(\mathcal{P})\frac{%
L_{\mathcal{Q}}}{A}\right) \geq 4A(\mathcal{P})^{2}F(\gamma )
\end{equation*}%
The previous Lemma guarantees that the left hand side converges to $0$ for
some subsequence $t_{j}\rightarrow T$. Since $F$ is a non-negative
functional we have also $F(t_{j})\rightarrow 0$ as $t_{j}\rightarrow T$. Let 
$\eta _{j}$ be the normalized curve $\eta _{j}=\displaystyle\sqrt{\frac{A(%
\mathcal{P})}{A}}\gamma (\cdot ,t_{j})$. Using the same technique presented
in \textbf{\cite{gage2}} one can show that the curves $\eta _{j}$ lie in one
same bounded region of the plane and then. Since $F$ satisfies property 
\textbf{\textit{(3)}} of Theorem \ref{teo2}, the region $H_{j}$ enclosed by $\eta
_{j} $ converges in the Hausdorff topology to the unit $\mathcal{P}$-disc.
It follows that $\displaystyle\frac{L_{\mathcal{Q}}^{2}(t_{j})}{A(t_{j})}$
converges to $4A(\mathcal{P})$ as $t_{j}\rightarrow T$. Since $\displaystyle%
\frac{L_{\mathcal{Q}}^{2}(t)}{A(t)}$ is nonincreasing the convergence holds,
in fact, for every value of the parameter and we have the desired result.
\end{proof}

\section{\label{SecExist}Existence of the minkowskian curvature flow}

The final step is to prove that the minkowskian curvature flow in fact
exists and continues until the area enclosed by the curves converges to
zero. We now establish:

\begin{lemma}
\label{lemma7} Let $k:[0,2\pi ]\rightarrow \mathbb{R}$ be a $C^{1}$ positive 
$2\pi $-periodic function. Then, $k$ is the Minkowski curvature of a simple
closed strictly convex $C^{2}$ plane curve if and only if 
\begin{equation}
\int_{0}^{2\pi }\frac{a(\theta )+a^{\prime \prime }(\theta )}{k(\theta )}%
\sin \theta \ d\theta =\int_{0}^{2\pi }\frac{a(\theta )+a^{\prime \prime
}(\theta )}{k(\theta )}\cos \theta \ d\theta =0  \label{eqn:16}
\end{equation}
\end{lemma}

\begin{proof}
Suppose first that $\gamma :[0,2\pi ]\rightarrow \mathbb{R}$ is a closed $%
C^{2}$ curve whose curvature is given by $k$. As $\displaystyle%
\int_{0}^{2\pi }\gamma ^{\prime }(\theta )d\theta =0$ we have 
\begin{equation*}
0=\int_{0}^{2\pi }\lambda (\theta )q(\theta )d\theta =\int_{0}^{2\pi }\frac{%
[p(\theta ),p^{\prime }(\theta )]}{k(\theta )a(\theta )}(-\sin \theta ,\cos
\theta )d\theta
\end{equation*}%
And then the desired equalities comes from equation \ref{eqn:Mink}.

On the other hand if $k$ is a $C^{1}$ positive $2\pi $-periodic function
such that (\ref{eqn:16}) holds we can define%
\begin{equation*}
\gamma (\theta )=\left( -\int_{0}^{\theta }\frac{a(\sigma )+a^{\prime \prime
}(\sigma )}{k(\sigma )}\sin \sigma \ d\sigma ,\int_{0}^{\theta }\frac{%
a(\sigma )+a^{\prime \prime }(\sigma )}{k(\sigma )}\cos \sigma \ d\sigma
\right)
\end{equation*}%
which is clearly a closed $C^{2}$ curve. Furthermore,%
\begin{equation*}
\gamma ^{\prime }(\theta )=\frac{a(\theta )+a^{\prime \prime }(\theta )}{%
k(\theta )}(-\sin \theta ,\cos \theta )=\frac{a(a+a^{\prime \prime })(\theta
)}{k(\theta )}q(\theta )=\frac{[p(\theta ),p^{\prime }(\theta )]}{k(\theta )}%
q(\theta )
\end{equation*}%
and then the Minkowski curvature of $\gamma $ is precisely $k$. To complete
the proof notice that $\gamma $ is simple as long as its Gauss map is
injective.
\end{proof}

Now, inspired by Lemma \ref{lemmapde} we will see how the solution to the
curvature motion emerges from the solution of a parabolic differential
equation. From now on, we use $t$ for the time parameter which is
independent with $\theta $.

\begin{teo}
\label{teo4} Consider a function $k:S^{1}\times \lbrack 0,T)\rightarrow 
\mathbb{R}$, such that $k\in C^{2+\alpha ,1+\alpha }(S^{1}\times \lbrack
0,T-\epsilon ])$ for all $\epsilon >0$, satisfying the evolution equation: 
\begin{equation}
\frac{\partial k}{\partial t}=\frac{a}{a+a^{\prime \prime }}k^{2}\frac{%
\partial ^{2}k}{\partial \theta ^{2}}+\frac{2a^{\prime }}{a+a^{\prime \prime
}}k^{2}\frac{\partial k}{\partial \theta }+k^{3}  \label{eqn:17}
\end{equation}%
with initial value $k(\theta ,0)=\varphi (\theta )$ where $\varphi $ is a
strictly positive $C^{1+\alpha }$ function such that: 
\begin{equation*}
\int_{0}^{2\pi }\frac{a(\theta )+a^{\prime \prime }(\theta )}{\varphi
(\theta )}\sin \theta \ d\theta =\int_{0}^{2\pi }\frac{a(\theta )+a^{\prime
\prime }(\theta )}{\varphi (\theta )}\cos \theta \ d\theta =0
\end{equation*}%
Using this function (whose short term existence and uniqueness are
guaranteed by standard theory on parabolic equations) one can build the
family of curves on parameter $t$: 
\begin{align*}
F(\theta ,t)& =\left( -\int_{0}^{\theta }\frac{a(\sigma )+a^{\prime \prime
}(\sigma )}{k(\sigma ,t)}\sin \sigma \ d\sigma -\int_{0}^{t}a(0)k(0,s)\
ds,\right. \\
& \left. \int_{0}^{\theta }\frac{a(\sigma )+a^{\prime \prime }(\sigma )}{%
k(\sigma ,t)}\cos \sigma \ d\sigma -\int_{0}^{t}a(0)\frac{\partial k}{%
\partial \sigma }(0,s)+a^{\prime }(0)k(0,s)\ ds\right)
\end{align*}%
for which the following holds:

\begin{description}
\item[(a)] for each fixed $t$ the map $\theta \mapsto F(\theta,t)$ is a
simple closed strictly convex curve parameterized as usual (the tangent
vector at $\theta$ points in the $q(\theta)$ direction) whose Minkowski
curvature is given by $\theta \mapsto k(\theta,t)$.

\item[(b)] $\displaystyle\frac{\partial F}{\partial t}(\theta ,t)=-k(\theta
,t)p(\theta )-a^{2}(\theta )\frac{\partial k}{\partial \theta }(\theta
,t)q(\theta )$
\end{description}
\end{teo}

\begin{proof}
For each fixed $t$ the curve $\theta \mapsto F(\theta ,t)$ is, up to a
translation, built as in Lemma \ref{lemma7}, and is clearly parameterized as
usual. So let's begin proving that $k(\theta ,t)$ is a strictly positive
function. Define 
\begin{equation*}
k_{\mathrm{MIN}}(t)=\inf_{[0,2\pi ]}k(\theta ,t)
\end{equation*}%
and notice that $k_{\mathrm{MIN}}$ is a continuous function which is
positive when $t=0$ (by the initial value conditions and compactness). We
claim that $k_{\mathrm{MIN}}$ is bounded from below by $k_{\mathrm{MIN}}(0)$%
. In fact, suppose there exists $t\in (0,T)$ such that $0<k_{\mathrm{MIN}%
}(t)=\delta <k_{\mathrm{MIN}}(0)$ and take $t_{0}=\inf k_{\mathrm{MIN}%
}^{-1}(\delta )$. Since $k_{\mathrm{MIN}}^{-1}(\delta )$ is a closed set we
have $t_{0}\in (0,T)$. By compactness the function $\theta \mapsto k(\theta
,t_{0})$ assumes the value $\delta $ for some $\theta _{0}\in \lbrack 0,2\pi
]$. Then, 
\begin{equation}
\frac{\partial k}{\partial t}(\theta _{0},t_{0})\leq 0;\ \ \frac{\partial k}{%
\partial \theta }(\theta _{0},t_{0})=0;\ \ \mathrm{and}\ \frac{\partial ^{2}k%
}{\partial \theta ^{2}}(\theta _{0},t_{0})\geq 0  \label{eqn:18}
\end{equation}%
For the first inequality observe that the function $t\mapsto k(\theta
_{0},t) $ must be nonincreasing by the left near $t_{0}$, otherwise the
definition of $t_{0}$ would be contradicted. The last two relations emerge
from the fact that $\theta _{0}$ is a minimum of the function $\theta
\mapsto k(\theta ,t_{0})$.

Finally, (\ref{eqn:18}) and $k(\theta _{0},t_{0})=\delta >0$ contradict the
assumption that $k$ satisfies (\ref{eqn:17}), as long as $\frac{a}{%
a+a^{\prime \prime }}>0$. This proves the claim and as consequence we have
that $k$ is strictly positive.

Our next step is to prove that, for each $t$, we have 
\begin{equation*}
\int_{0}^{2\pi }\frac{a(\sigma )+a^{\prime \prime }(\sigma )}{k(\sigma ,t)}%
\sin \sigma \ d\sigma =\int_{0}^{2\pi }\frac{a(\sigma )+a^{\prime \prime
}(\sigma )}{k(\sigma ,t)}\cos \sigma \ d\sigma =0
\end{equation*}%
By the hypothesis this is true for $t=0$. So it's enough to prove that the
derivatives of the functions $t\mapsto -\displaystyle\int_{0}^{2\pi }\frac{%
a(\sigma )+a^{\prime \prime }(\sigma )}{k(\sigma ,t)}\sin \sigma \ d\sigma $
and $t\mapsto \displaystyle\int_{0}^{2\pi }\frac{a(\sigma )+a^{\prime \prime
}(\sigma )}{k(\sigma ,t)}\cos \sigma \ d\sigma $ vanish identically. Using (%
\ref{eqn:17}) and integration by parts we calculate 
\begin{align*}
\frac{d}{dt}\left( -\int_{0}^{2\pi }\frac{a(\sigma )+a^{\prime \prime
}(\sigma )}{k(\sigma ,t)}\sin \sigma \ d\sigma \right) & =\int_{0}^{2\pi }%
\frac{1}{k^{2}}\frac{\partial k}{\partial t}(a+a^{\prime \prime })\sin
\sigma \ d\sigma = \\
& =\int_{0}^{2\pi }\sin \sigma \frac{\partial }{\partial \sigma }\left( a%
\frac{\partial k}{\partial \sigma }\right) +\sin \sigma \frac{\partial }{%
\partial \sigma }\left( a^{\prime }k\right) +ak\sin \sigma \ d\sigma = \\
& =a\frac{\partial k}{\partial \theta }\sin \theta \bigg|_{0}^{2\pi
}+a^{\prime }k\sin \theta \bigg|_{0}^{2\pi }-ak\cos \theta \bigg|_{0}^{2\pi
}=0
\end{align*}%
where the last equality comes from the fact that all the involved functions
are $2\pi $-periodic. We do the same for the other function and then Lemma %
\ref{lemma7} yields \textbf{(a)}. \newline

For \textbf{(b)} we calculate the time derivatives of each component using,
again, integration by parts and (\ref{eqn:17}). For the first component we
have: 
\begin{multline*}
\frac{d}{dt}\left( -\int_{0}^{\theta }\frac{a(\sigma )+a^{\prime \prime
}(\sigma )}{k(\sigma ,t)}\sin \sigma \ d\sigma -\int_{0}^{t}a(0)k(0,s)\
ds\right) = \\
=\int_{0}^{\theta }\frac{1}{k^{2}}\frac{\partial k}{\partial t}(a+a^{\prime
\prime })\sin \sigma \ d\sigma -a(0)k(0,t)= \\
=\int_{0}^{\theta }\sin \sigma \frac{\partial }{\partial \sigma }\left( a%
\frac{\partial k}{\partial \sigma }\right) +\sin \sigma \frac{\partial }{%
\partial \sigma }\left( a^{\prime }k\right) +ak\sin \sigma \ d\sigma
-a(0)k(0,t)= \\
=a\frac{\partial k}{\partial \sigma }\sin \sigma \bigg|_{0}^{\theta
}+a^{\prime }k\sin \sigma \bigg|_{0}^{\theta }-ak\cos \sigma \bigg|%
_{0}^{\theta }-a(0)k(0,t)= \\
=a(\theta )\frac{\partial k}{\partial \theta }(\theta ,t)\sin \theta
+a^{\prime }(\theta )k(\theta ,t)\sin \theta -a(\theta )k(\theta ,t)\cos
\theta
\end{multline*}%
And for the second: 
\begin{multline*}
\frac{d}{dt}\left( \int_{0}^{\theta }\frac{a(\sigma )+a^{\prime \prime
}(\sigma )}{k(\sigma ,t)}\cos \sigma \ d\sigma -\int_{0}^{t}a(0)\frac{%
\partial k}{\partial \sigma }(0,s)+a^{\prime }(0)k(0,s)\ ds\right) = \\
=-\int_{0}^{\theta }a\frac{\partial ^{2}k}{\partial \sigma ^{2}}\cos \sigma
+2a^{\prime }\frac{\partial k}{\partial \theta }\cos \sigma +(a+a^{\prime
\prime })k\cos \sigma \ d\sigma -a(0)\frac{\partial k}{\partial \theta }%
(0,t)-a^{\prime }(0)k(0,t)= \\
=-a\frac{\partial k}{\partial \sigma }\cos \sigma \bigg|_{0}^{\theta
}-a^{\prime }k\cos \sigma \bigg|_{0}^{\theta }-ak\sin \sigma \bigg|%
_{0}^{\theta }-a(0)\frac{\partial k}{\partial \theta }(0,t)-a^{\prime
}(0)k(0,t)= \\
=-a(\theta )\frac{\partial k}{\partial \theta }(\theta ,t)\cos \theta
-a^{\prime }(\theta )k(\theta ,t)\cos \theta -a(\theta )k(\theta ,t)\sin
\theta
\end{multline*}%
Therefore, 
\begin{multline*}
\frac{\partial F}{\partial t}(\theta ,t)=-k(\theta ,t).\left( a(\theta )\cos
\theta -a^{\prime }(\theta )\sin \theta ,a(\theta )\sin \theta +a^{\prime
}(\theta )\cos \theta \right) -a(\theta )\frac{\partial k}{\partial \theta }%
(\theta ,t).(-\sin \theta ,\cos \theta )= \\
=-k(\theta ,t)p(\theta )-a^{2}(\theta )\frac{\partial k}{\partial \theta }%
(\theta ,t)q(\theta )
\end{multline*}%
and this concludes the proof.
\end{proof}

By changing the space parameter one can make the tangential component vanish
while keeping the shape of the curves. For this reason Theorem \ref{teo4}
yields the desired Minkowski curvature flow stated in (\ref{eqn:7}). Notice
that it follows also that the curves remain simple and strictly convex along
the motion.

To show that the solution continues until the area enclosed by the curves
converges to zero we prove that the curvature and its derivatives remain
bounded as long as the area is bounded away from zero. Let us begin with a
Lemma that is independent of the flow.

\begin{defi}
\label{defi1} Consider a curve parameterized by the usual $\theta $ and with
Minkowski curvature $k$. We define the minkowskian median curvature $k^{\ast
}$ for the curve as the supremum of all values $x$ for which we have $%
k(\theta )>x$ on some interval of length $\pi $.
\end{defi}

\begin{lemma}
\label{lemma8} Let $\gamma :[0,2\pi ]\rightarrow \mathbb{R}^{2}$ be a curve
in the Minkowski plane which is simple, closed and convex. Denote, as usual,
the $\mathcal{Q}$-length and the enclosed area by $L_{\mathcal{Q}}$ and $A$
respectively. Then, 
\begin{equation*}
k^{\ast }\leq C\frac{L_{\mathcal{Q}}}{A}
\end{equation*}%
for some constant $C$ that doesn't depends on the curve.
\end{lemma}

\begin{proof}
Writing as usual $\gamma ^{\prime }(\theta )=\lambda (\theta )q(\theta )$ we
have that the $\mathcal{Q}$-length is given by 
\begin{equation*}
s(\theta )=\int_{0}^{\theta }\lambda (\sigma )d\sigma
\end{equation*}%
and the euclidean length is given by 
\begin{equation*}
s_{E}(\theta )=\int_{0}^{\theta }\lambda (\sigma )|q(\sigma )|d\sigma
\end{equation*}%
where $|\cdot |$ is the euclidean norm. Put $q_{0}=\max_{[0,2\pi ]}|q(\theta
)|$. Denoting by $L$ the euclidean length of $\gamma $ is easy to see that $%
L\leq q_{0}L_{\mathcal{Q}}$. Furthermore, denoting the euclidean curvature
by $k_{E}$ we have $k_{E}(\theta )=k(\theta )\left( |q(\theta )|[p(\theta
),p^{\prime }(\theta )]\right) ^{-1}$.

If $0<B<k^{\ast }$ we can take an interval $(\theta _{0},\theta _{0}+\pi )$
in which $k>B$. Moreover, we know that the area is bounded by any usual
width times $L/2$. Then 
\begin{align*}
A& \leq \frac{L}{2}\left\vert \int_{\theta _{0}}^{\theta _{0}+\pi }\frac{%
\sin (\theta _{0}-\theta )}{k_{E}(\theta )}d\theta \right\vert =\frac{L}{2}%
\left\vert \int_{\theta _{0}}^{\theta _{0}+\pi }\frac{|q(\theta )|[p(\theta
),p^{\prime }(\theta )]\sin (\theta _{0}-\theta )}{k(\theta )}d\theta
\right\vert \leq \\
& \leq \frac{q_{0}^{2}L_{\mathcal{Q}}}{2}\max_{\theta \in \lbrack 0,2\pi
]}[p(\theta ),p^{\prime }(\theta )]\int_{\theta _{0}}^{\theta _{0}+\pi
}\left\vert \frac{\sin (\theta _{0}-\theta )}{k(\theta )}\right\vert d\theta
\leq \frac{q_{0}^{2}L_{\mathcal{Q}}}{2}\max_{\theta \in \lbrack 0,2\pi
]}[p(\theta ),p^{\prime }(\theta )]\frac{2}{B}= \\
& =q_{0}^{2}\max_{\theta \in \lbrack 0,2\pi ]}[p(\theta ),p^{\prime }(\theta
)]\frac{L_{\mathcal{Q}}}{B}
\end{align*}%
Making $B\rightarrow k^{\ast }$ and taking $C=q_{0}^{2}\max_{\theta \in
\lbrack 0,2\pi ]}[p(\theta ),p^{\prime }(\theta )]$ conclude the proof. Note
carefully that $C$ only depends on the set $\mathcal{P}$ chosen as unit ball
of our Minkowski plane.
\end{proof}

It is natural to denote by $k^{\ast }(t)$ the minkowskian median curvature
of the flow curve $\theta \mapsto F(\theta ,t)$. Notice that if the areas
enclosed by the curves are bounded from below on $[0,T)$ by some number $c>0$
then the median curvatures have an uniform upper bound on $[0,T)$.

\begin{prop}
\label{prop3} If $k^*(t)$ is bounded on $[0,T)$ then $\displaystyle%
\int_0^{2\pi}\left(a(\theta)+a^{\prime \prime }(\theta)\right)a(\theta)\log
k(\theta,t)d\theta$ is also bounded on $[0,T)$.
\end{prop}

\begin{proof}
First, adopting an easier notation we calculate 
\begin{align*}
\frac{d}{dt}\left( \int_{0}^{2\pi }a(a+a^{\prime \prime })\log k\ d\theta
\right) & =\int_{0}^{2\pi }\frac{a(a+a^{\prime \prime })}{k}\frac{\partial k%
}{\partial t}d\theta = \\
& =\int_{0}^{2\pi }a^{2}k\frac{\partial ^{2}k}{\partial \theta ^{2}}%
+2aa^{\prime }k\frac{\partial k}{\partial \theta }+a\left( a+a^{\prime
\prime }\right) k^{2}d\theta = \\
& =\int_{0}^{2\pi }-\left( a\frac{\partial k}{\partial \theta }\right)
^{2}+(ak)^{2}+aa^{\prime \prime }k^{2}\ d\theta = \\
& =\int_{0}^{2\pi }(ak)^{2}-\left( \frac{\partial (ak)}{\partial \theta }%
\right) ^{2}d\theta 
\end{align*}%
here we used integration by parts and the evolution equation. A version of
the Wirtinger's inequality states that if $f:[a,b]\rightarrow \mathbb{R}$ is
a $C^{1}$ function such that $b-a\leq \pi $ and $f(a)=f(b)=0$ then 
\begin{equation*}
\int_{a}^{b}f^{2}dx\leq \int_{a}^{b}\left( f^{\prime }\right) ^{2}dx
\end{equation*}%
and we will use this result to estimate the above integral. Fix $t$ and
consider the set $A\subseteq \lbrack 0,2\pi ]$ given by $A=\{\theta \in
\lbrack 0,2\pi ]\mid k(\theta ,t)>k^{\ast }(t)\}$. By the definition of $%
k^{\ast }$ we note that $A$ is an at most countable union of disjoint
intervals $I_{j}$ with $|I_{j}|\leq \pi $ for each $j$ and such that $%
k(\theta ,t)=k^{\ast }(t)$ on its endpoints. So, applying the Wirtinger's
inequality to the restriction of the function $\theta \mapsto a(\theta
)k(\theta ,t)-a(\theta )k^{\ast }(\theta )$ to an interval $I_{j}$ yields 
\begin{equation*}
\int_{I_{j}}(ak)^{2}-2a^{2}kk^{\ast }+(ak^{\ast })^{2}d\theta \leq
\int_{I_{j}}\left( \frac{\partial (ak)}{\partial \theta }\right)
^{2}-2a^{\prime }k^{\ast }\frac{\partial (ak)}{\partial \theta }+(a^{\prime
}k^{\ast })^{2}d\theta 
\end{equation*}%
And then, 
\begin{equation*}
\int_{I_{j}}(ak)^{2}-\left( \frac{\partial (ak)}{\partial \theta }\right)
^{2}d\theta \leq 2k^{\ast }(t)\int_{I_{j}}a\left( a+a^{\prime \prime
}\right) kd\theta -k^{\ast }(t)^{2}\int_{I_{j}}\left( a^{\prime }\right)
^{2}+2aa^{\prime \prime }+a^{2}\ d\theta 
\end{equation*}%
Summating over $j$ yields the following estimate on $A$: 
\begin{align*}
\int_{A}(ak)^{2}-\left( \frac{\partial (ak)}{\partial \theta }\right)
^{2}d\theta & \leq 2k^{\ast }(t)\int_{0}^{2\pi }a\left( a+a^{\prime \prime
}\right) kd\theta -k^{\ast }(t)^{2}\int_{A}\left( a^{\prime }\right)
^{2}+2aa^{\prime \prime }+a^{2}\ d\theta = \\
& =-2k^{\ast }(t)\frac{dL_{\mathcal{Q}}}{dt}-k^{\ast
}(t)^{2}\int_{A}\left( a^{\prime }\right) ^{2}+2aa^{\prime \prime }+a^{2}\
d\theta \leq  \\
& \leq -2k^{\ast }(t)\frac{dL_{\mathcal{Q}}}{dt}+2\pi
k^{\ast }(t)^{2}\max_{[0,2\pi ]}\left\vert \left( a^{\prime }\right)
^{2}+2aa^{\prime \prime }+a^{2}\right\vert 
\end{align*}%
On $[0,2\pi ]-A$ we have the estimate 
\begin{equation*}
\int_{\lbrack 0,2\pi ]-A}(ak)^{2}-\left( \frac{\partial (ak)}{\partial
\theta }\right) ^{2}d\theta \leq \int_{\lbrack 0,2\pi ]-A}(ak)^{2}d\theta
\leq 2\pi k^{\ast }(t)^{2}\max_{[0,2\pi ]}a^{2}
\end{equation*}%
Suppose that $M>0$ is an upper bound for $k^{\ast }(t)$ on $[0,T)$. Then,
the above estimates yields 
\begin{equation*}
\frac{d}{dt}\left( \int_{0}^{2\pi }a(a+a^{\prime \prime })\log k\ d\theta
\right) \leq -2M\frac{dL_{\mathcal{Q}}}{dt}+2\pi M^{2}C_{0}
\end{equation*}%
For some constant $C_{0}>0$ that only depends on the unit $\mathcal{P}$-ball
chosen. Let $\displaystyle\int_{0}^{2\pi }a(a+a^{\prime \prime })\log k\
d\theta =C_{1}$ for $t=0$. We write 
\begin{align*}
\int_{0}^{2\pi }a(\theta )(a(\theta )+a^{\prime \prime }(\theta ))\log
k(\theta ,t)\ d\theta & =C_{1}+\int_{0}^{t}\left( \frac{d}{dt}\left(
\int_{0}^{2\pi }a(a+a^{\prime \prime })\log k\ d\theta \right) \right)
dt\leq  \\
& \leq C_{1}+\int_{0}^{t}-2M\frac{dL_{\mathcal{Q}}}{dt}+2\pi M^{2}C_{0}\
dt\leq  \\
& \leq C_{1}-2M\left( L_{\mathcal{Q}}(t)-L_{\mathcal{Q}}(0)\right) +2\pi
M^{2}C_{0}T\leq  \\
& \leq C_{1}+2ML_{\mathcal{Q}}(0)+2\pi M^{2}C_{0}T
\end{align*}%
and this completes the proof since the right side does not depends on $t$.
\end{proof}

\begin{lemma}
\label{lemma9} If $\displaystyle\int_0^{2\pi}a(\theta)(a(\theta)+a^{\prime
\prime }(\theta))\log k(\theta,t) \ d\theta$ is bounded on $[0,T)$, then for
any $\delta > 0$ there exists a constant $C$ such that if $k(\theta,t) > C$
on an interval $J$ (varying the parameter $\theta$) then we have necessarily 
$|J| \leq \delta$.
\end{lemma}

\begin{proof}
Fix $\delta >0$ and take $[b,c]\subseteq \lbrack 0,2\pi ]$ with length
greater than $\delta $. Suppose that $k(\theta ,t)>C$ on $[b,c]$ for some $t$%
. Remembering that $k_{\mathrm{MIN}}(0)$ is a lower bound for $k(\theta ,t)$
we have 
\begin{multline*}
\int_{0}^{2\pi }a(a+a^{\prime \prime })\log k(\theta ,t)d\theta = \\
=\int_{0}^{b}a(a+a^{\prime \prime })\log k(\theta ,t)d\theta
+\int_{b}^{c}a(a+a^{\prime \prime })\log k(\theta ,t)d\theta +\int_{c}^{2\pi
}a(a+a^{\prime \prime })\log k(\theta ,t)d\theta \geq \\
\geq \log \left( k_{\mathrm{MIN}}(0)\right) \int_{0}^{b}a(a+a^{\prime \prime
})\ d\theta +\delta \log C\int_{0}^{2\pi }a(a+a^{\prime \prime })d\theta
+\log \left( k_{\mathrm{MIN}}(0)\right) \int_{c}^{2\pi }a(a+a^{\prime \prime
})\ d\theta
\end{multline*}%
If $k_{\mathrm{MIN}}(0)\leq 1$, then 
\begin{multline*}
\int_{0}^{2\pi }a(a+a^{\prime \prime })\log k(\theta ,t)d\theta \geq \left[
\delta \log C+(2\pi +b-c)\log \left( k_{\mathrm{MIN}}(0)\right) \right]
\max_{[0,2\pi ]}a(a+a^{\prime \prime })\geq \\
\geq \left[ \delta \log C+(2\pi -\delta )\log \left( k_{\mathrm{MIN}%
}(0)\right) \right] \max_{[0,2\pi ]}a(a+a^{\prime \prime })
\end{multline*}%
Otherwise we have 
\begin{multline*}
\int_{0}^{2\pi }a(a+a^{\prime \prime })\log k(\theta ,t)d\theta \geq \\
\geq \delta \log C\int_{0}^{2\pi }a(a+a^{\prime \prime })d\theta +(2\pi
+b-c)\log \left( k_{\mathrm{MIN}}(0)\right) \min_{[0,2\pi ]}a(a+a^{\prime
\prime })\geq \\
\geq \left( \delta \log C\right) \max_{[0,2\pi ]}a(a+a^{\prime \prime })
\end{multline*}%
Both cases are contradictions when $C$ is sufficiently large since the left
side is bounded on $[0,T)$. This proves the result.
\end{proof}

\begin{lemma}
\label{lemma10} The function $t\mapsto \displaystyle\int_{0}^{2\pi }\left(
a(\theta )k(\theta ,t)\right) ^{2}-\left( \frac{\partial }{\partial \theta }%
(a(\theta )k(\theta ,t))\right) ^{2}d\theta $ is nondecreasing. In
par\-ti\-cu\-lar, one can find a constant $N\geq 0$ such that the inequality 
\begin{equation*}
\int_{0}^{2\pi }\left( \frac{\partial (ak)}{\partial \theta }\right)
^{2}d\theta \leq \int_{0}^{2\pi }(ak)^{2}d\theta +N
\end{equation*}%
holds on $[0,T)$.
\end{lemma}

\begin{proof}
We compute 
\begin{multline*}
\frac{d}{dt}\left( \int_{0}^{2\pi }\left( a(\theta )k(\theta ,t)\right)
^{2}-\left( \frac{\partial }{\partial \theta }(a(\theta )k(\theta
,t))\right) ^{2}d\theta \right) =\int_{0}^{2\pi }2a^{2}k\frac{\partial k}{%
\partial t}-2\frac{\partial (ak)}{\partial \theta }\frac{\partial ^{2}(ak)}{%
\partial \theta \partial t}\ d\theta = \\
=\int_{0}^{2\pi }2a^{2}k\frac{\partial k}{\partial t}+2a\frac{\partial
^{2}(ak)}{\partial \theta ^{2}}\frac{\partial k}{\partial t}d\theta
=\int_{0}^{2\pi }2a\frac{\partial k}{\partial t}\left( ak+\frac{\partial
^{2}(ak)}{\partial \theta ^{2}}\right) d\theta = \\
=\int_{0}^{2\pi }2a\frac{\partial k}{\partial t}\left( ak+a^{\prime \prime
}k+2a^{\prime }\frac{\partial k}{\partial \theta }+a\frac{\partial ^{2}k}{%
\partial \theta ^{2}}\right) \ d\theta =2\int_{0}^{2\pi }\frac{a(a+a^{\prime
\prime })}{k^{2}}\left( \frac{\partial k}{\partial t}\right) ^{2}d\theta
\geq 0
\end{multline*}%
\newline
and this proves the first claim. To find $N\geq 0$ with the desired property
it is enough to take any positive number greater then the value of the
function when $t=0$.
\end{proof}

\begin{prop}
\label{prop4} If $\displaystyle\int_0^{2\pi}a(\theta)(a(\theta)+a^{\prime
\prime }(\theta))\log k(\theta,t) \ d\theta$ is bounded on $[0,T)$, then $%
k(\theta,t)$ has an upper bound on $S^1\times [0,T)$.
\end{prop}

\begin{proof}
We shall find an upper bound for the function $t\mapsto k_{\mathrm{MAX}}(t)$%
. Fix $t\in \lbrack 0,T)$ and let $\theta _{0}\in \lbrack 0,2\pi ]$ such
that $k(\theta _{0},t)=k_{\mathrm{MAX}}(t)$. Denote $\min_{[0,2\pi ]}a=a_{0}$
and $\max_{[0,2\pi ]}a=a_{1}$, choose $0<\delta <\displaystyle\frac{a_{0}^{2}%
}{2\pi a_{1}^{2}}$ and let $C$ be as in Lemma \ref{lemma9}. Therefore, we
can take $b\in \lbrack 0,2\pi ]$ such that $k(b,t)\leq C$ and $0<|b-\theta
_{0}|\leq \delta $. Changing the parameter if necessary we can assume $%
b<\theta _{0}$. Moreover, let $N>0$ be as in Lemma \ref{lemma10}. Using the
Holder's inequality we calculate 
\begin{multline*}
k_{\mathrm{MAX}}(t)=\frac{1}{a(\theta _{0})}a(\theta _{0})k(\theta _{0},t)=%
\frac{1}{a(\theta _{0})}a(b)k(b,t)+\frac{1}{a(\theta _{0})}\int_{b}^{\theta
_{0}}\frac{\partial (ak)}{\partial \theta }d\theta \leq \\
\leq \frac{Ca(b)}{a(\theta _{0})}+\frac{\sqrt{\delta }}{a(\theta _{0})}%
\left( \int_{b}^{\theta _{0}}\left( \frac{\partial (ak)}{\partial \theta }%
\right) ^{2}d\theta \right) ^{1/2}\leq \frac{Ca(b)}{a(\theta _{0})}+\frac{%
\sqrt{\delta }}{a(\theta _{0})}\left( \int_{0}^{2\pi }(ak)^{2}d\theta
+N\right) ^{1/2}\leq \\
\leq \frac{Ca(b)}{a(\theta _{0})}+\frac{\sqrt{\delta }}{a(\theta _{0})}\sqrt{%
2\pi }a_{1}k_{\mathrm{MAX}}(t)+\frac{\sqrt{\delta N}}{a(\theta _{0})}\leq 
\frac{Ca_{1}}{a_{0}}+\frac{a_{1}\sqrt{2\pi \delta }}{a_{0}}k_{\mathrm{MAX}%
}(t)+\frac{\sqrt{\delta N}}{a_{0}}
\end{multline*}%
Then we have 
\begin{equation*}
k_{\mathrm{MAX}}(t)\left( 1-\frac{a_{1}}{a_{0}}\sqrt{2\pi \delta }\right)
\leq \frac{Ca_{1}+\sqrt{\delta N}}{a_{0}}
\end{equation*}%
And, finally, by the assumption on $\delta $, 
\begin{equation*}
k_{\mathrm{MAX}}(t)\leq \frac{Ca_{1}+\sqrt{\delta N}}{a_{0}-a_{1}\sqrt{2\pi
\delta }}
\end{equation*}%
Since the right side does not depends on $t$ we have the desired.
\end{proof}

Combining these lemmas and propositions yields immediately the following
theorem:

\begin{teo}
\label{teo5} Let $A(t)$ denote the area enclosed by the curve $\theta
\mapsto k(\theta ,t)$. If $A(t)$ admits a strictly positive lower bound on $%
[0,T)$, then $k(\theta ,t)$ is uniformly bounded on $S^{1}\times \lbrack
0,T) $.
\end{teo}

We now turn our attention to prove that the derivatives of $k$ remain
bounded as long as $k$ is bounded.

\begin{prop}
\label{prop5} If $k$ is bounded on $S^1\times [0,T)$, then $\displaystyle%
\frac{\partial k}{\partial\theta}$ is also bounded on $S^1\times [0,T)$.
\end{prop}

\begin{proof}
Consider the function $f:S^{1}\times \lbrack 0,T)$ given by $f=\displaystyle %
a^{2}\left( \theta \right) \frac{\partial k}{\partial \theta }e^{ct}$, where 
$c$ is to be chosen later. After some calculations we see that $f$ is a
solution of the second order parabolic equation 
\begin{equation*}
\frac{\partial f}{\partial t}=\left( c+3k^{2}\right) f-k^{2}\frac{2a^{\prime
}}{a+a^{\prime \prime }}\frac{\partial f}{\partial \theta }+\frac{\partial }{%
\partial \theta }\left( k^{2}\frac{a}{a+a^{\prime \prime }}\frac{\partial f}{%
\partial \theta }\right)
\end{equation*}%
Now, taking $c\leq -3\max_{S^{1}\times \lbrack 0,T)}k^{2}$ we can bound $f$
using the maximum principle. It follows that $\displaystyle\frac{\partial k}{%
\partial \theta }$ is also bounded for finite time.
\end{proof}

To prove that the second spatial derivative is bounded we follow, again, the
method used in \textbf{\cite{gage3}}.

\begin{lemma}
\label{lemma11} Define the function $\xi :[0,T)\rightarrow \mathbb{R}$ by 
\begin{equation*}
\xi (t)=\int_{0}^{2\pi }\left( \frac{\partial ^{2}k}{\partial \theta ^{2}}%
\right) ^{4}d\theta
\end{equation*}%
If $k$ is bounded in $S^{1}\times \lbrack 0,T)$ then the function $\xi $ is
also bounded in $[0,T)$.
\end{lemma}

\begin{proof}
Let us denote for simplicity $F=\displaystyle\frac{a}{a+a^{\prime \prime }}$
and $G=\displaystyle\frac{2a^{\prime }}{a+a^{\prime \prime }}$. Using
integration by parts and the evolution equation we compute 
\begin{align*}
\frac{d\xi }{dt}& =\frac{d}{dt}\left( \int_{0}^{2\pi }\left( \frac{\partial
^{2}k}{\partial \theta ^{2}}\right) ^{4}d\theta \right) =4\int_{0}^{2\pi
}\left( \frac{\partial ^{2}k}{\partial \theta ^{2}}\right) ^{3}\frac{%
\partial }{\partial t}\left( \frac{\partial ^{2}k}{\partial \theta ^{2}}%
\right) d\theta = \\
& =-12\int_{0}^{2\pi }\left( \frac{\partial ^{2}k}{\partial \theta ^{2}}%
\right) ^{2}\frac{\partial ^{3}k}{\partial \theta ^{3}}\frac{\partial }{%
\partial \theta }\left( k^{2}F\frac{\partial ^{2}k}{\partial \theta ^{2}}%
+k^{2}G\frac{\partial k}{\partial \theta }+k^{3}\right) d\theta = \\
& =12\int_{0}^{2\pi }k\frac{\partial k}{\partial \theta }\left( \frac{%
\partial ^{2}k}{\partial \theta ^{2}}\right) ^{2}\frac{\partial ^{3}k}{%
\partial \theta ^{3}}\left( -2F\frac{\partial ^{2}k}{\partial \theta ^{2}}-k%
\frac{\partial G}{\partial \theta }-3k\right) -\frac{\partial F}{\partial
\theta }k^{2}\left( \frac{\partial ^{2}k}{\partial \theta ^{2}}\right) ^{3}%
\frac{\partial ^{3}k}{\partial \theta ^{3}}- \\
& -2Gk\left( \frac{\partial k}{\partial \theta }\right) ^{2}\left( \frac{%
\partial ^{2}k}{\partial \theta ^{2}}\right) ^{2}\frac{\partial ^{3}k}{%
\partial \theta ^{3}}-Gk^{2}\left( \frac{\partial ^{2}k}{\partial \theta ^{2}%
}\right) ^{3}\frac{\partial ^{3}k}{\partial \theta ^{3}}-Fk^{2}\left( \frac{%
\partial ^{2}k}{\partial \theta ^{2}}\right) ^{2}\left( \frac{\partial ^{3}k%
}{\partial \theta ^{3}}\right) ^{2}d\theta
\end{align*}%
Now put $C_{1}=\min_{S^{1}}F$. We have 
\begin{equation*}
-Fk^{2}\left( \frac{\partial ^{2}k}{\partial \theta ^{2}}\right) ^{2}\left( 
\frac{\partial ^{3}k}{\partial \theta ^{3}}\right) ^{2}\leq
-C_{1}k^{2}\left( \frac{\partial ^{2}k}{\partial \theta ^{2}}\right)
^{2}\left( \frac{\partial ^{3}k}{\partial \theta ^{3}}\right) ^{2}
\end{equation*}%
Notice that $C_{1}>0$ and remember that if $k$ is bounded then $\displaystyle%
\frac{\partial k}{\partial \theta }$ is also bounded. Choosing $\displaystyle%
\epsilon =\frac{16}{C_{1}}$ we can use the inequality $ab\leq \displaystyle%
\frac{4}{\epsilon }a^{2}+\epsilon b^{2}$ to estimate the other terms of the
integral as follows:

\noindent $\mathrm{1.}\ \displaystyle\left( k\frac{\partial ^{2}k}{\partial
\theta ^{2}}\frac{\partial ^{3}k}{\partial \theta ^{3}}\right) \left[ \left(
-2F\frac{\partial ^{2}k}{\partial \theta ^{2}}-k\frac{\partial G}{\partial
\theta }-3k\right) \frac{\partial k}{\partial \theta }\frac{\partial ^{2}k}{%
\partial \theta ^{2}}\right] \leq \frac{4}{\epsilon }\left( k\frac{\partial
^{2}k}{\partial \theta ^{2}}\frac{\partial ^{3}k}{\partial \theta ^{3}}%
\right) ^{2}+C_{2}\left( \left( \frac{\partial ^{2}k}{\partial \theta ^{2}}%
\right) ^{4}+\left( \frac{\partial ^{2}k}{\partial \theta ^{2}}\right)
^{2}\right) ,$ where $C_{2}=\epsilon \displaystyle\max_{S^{1}\times \lbrack
0,T)}\left( \frac{\partial k}{\partial \theta }\right) ^{2}\left(
12F^{2}+3k^{2}\left( \frac{\partial G}{\partial \theta }\right)
^{2}+27k^{2}\right) $. Here we also used the inequality\newline
$\left( a+b+c\right) ^{2}\leq 3\left( a^{2}+b^{2}+c^{2}\right) $.

\noindent $\mathrm{2.}\ \displaystyle\left( k\frac{\partial ^{2}k}{\partial
\theta ^{2}}\frac{\partial ^{3}k}{\partial \theta ^{3}}\right) \left[ -\frac{%
\partial F}{\partial \theta }k\left( \frac{\partial ^{2}k}{\partial \theta
^{2}}\right) ^{2}\right] \leq \frac{4}{\epsilon }\left( k\frac{\partial ^{2}k%
}{\partial \theta ^{2}}\frac{\partial ^{3}k}{\partial \theta ^{3}}\right)
^{2}+C_{3}\left( \frac{\partial ^{2}k}{\partial \theta ^{2}}\right) ^{4},$
where $C_{3}=\epsilon \displaystyle\max_{S^{1}}\left( \frac{\partial F}{%
\partial \theta }\right) ^{2}\max_{S^{1}\times \lbrack 0,T)}k^{2}.$

\noindent $\mathrm{3.}\ \displaystyle\left( k\frac{\partial ^{2}k}{\partial
\theta ^{2}}\frac{\partial ^{3}k}{\partial \theta ^{3}}\right) \left[
-2G\left( \frac{\partial k}{\partial \theta }\right) ^{2}\frac{\partial ^{2}k%
}{\partial \theta ^{2}}\right] \leq \frac{4}{\epsilon }\left( k\frac{%
\partial ^{2}k}{\partial \theta ^{2}}\frac{\partial ^{3}k}{\partial \theta
^{3}}\right) ^{2}+C_{4}\left( \frac{\partial ^{2}k}{\partial \theta ^{2}}%
\right) ^{2},$ where $C_{4}=4\epsilon \displaystyle\max_{S^{1}\times \lbrack
0,T)}\left( \frac{\partial k}{\partial \theta }\right)
^{4}\max_{S^{1}}G^{2}. $

\noindent $\mathrm{4.}\ \displaystyle\left( k\frac{\partial ^{2}k}{\partial
\theta ^{2}}\frac{\partial ^{3}k}{\partial \theta ^{3}}\right) \left[
-Gk\left( \frac{\partial ^{2}k}{\partial \theta ^{2}}\right) ^{2}\right]
\leq \frac{4}{\epsilon }\left( k\frac{\partial ^{2}k}{\partial \theta ^{2}}%
\frac{\partial ^{3}k}{\partial \theta ^{3}}\right) ^{2}+C_{5}\left( \frac{%
\partial ^{2}k}{\partial \theta ^{2}}\right) ^{4},$ where $C_{5}=\epsilon %
\displaystyle\max_{S^{1}}G^{2}\max_{S^{1}\times \lbrack 0,T)}k^{2}.$

Combining these estimates and writing $C_{6}=12(C_{2}+C_{3}+C_{5})$ and $%
C_{7}=12(C_{2}+C_{4})$ we have 
\begin{equation*}
\frac{d\xi }{dt}\leq C_{6}\int_{0}^{2\pi }\left( \frac{\partial ^{2}k}{%
\partial \theta ^{2}}\right) ^{4}d\theta +C_{7}\int_{0}^{2\pi }\left( \frac{%
\partial ^{2}k}{\partial \theta ^{2}}\right) ^{2}d\theta
\end{equation*}%
Holder's inequality and the inequality $\sqrt{A}\leq \displaystyle\frac{A+1}{%
2}$ (if $A\geq 0$) yield%
\begin{align*}
\frac{d\xi }{dt}& \leq C_{6}\int_{0}^{2\pi }\left( \frac{\partial ^{2}k}{%
\partial \theta ^{2}}\right) ^{4}d\theta +C_{7}\sqrt{2\pi }\left(
\int_{0}^{2\pi }\left( \frac{\partial ^{2}k}{\partial \theta ^{2}}\right)
^{4}d\theta \right) ^{1/2}= \\
& =C_{6}\xi +C_{7}\sqrt{2\pi }\sqrt{\xi }\leq C_{6}\xi +C_{7}\sqrt{2\pi }%
\left( \frac{\xi +1}{2}\right) =C_{8}\xi +C_{9},
\end{align*}%
with $C_{8}=C_{6}+\displaystyle\frac{C_{7}\sqrt{2\pi }}{2}$ and $C_{9}=%
\displaystyle\frac{C_{7}\sqrt{2\pi }}{2}$.

By the Gronwall's inequality we have immediately that $\xi $ is bounded for
finite time. This completes the proof.
\end{proof}

\begin{coro}
\label{coro1} The function $t \mapsto \displaystyle\int_0^{2\pi}\left(\frac{%
\partial^2k}{\partial\theta^2}\right)^2d\theta$ is bounded in $[0,T)$.
\end{coro}

\begin{proof}
This is immediate by the Holder's inequality.
\end{proof}

\begin{lemma}
\label{lemma12} The function $\beta :[0,T)\rightarrow \mathbb{R}$ given by 
\begin{equation*}
\beta (t)=\int_{0}^{2\pi }\left( \frac{\partial ^{3}k}{\partial \theta ^{3}}%
\right) ^{2}d\theta
\end{equation*}%
is bounded provided $k$ is bounded in $S^{1}\times \lbrack 0,T)$.
\end{lemma}

\begin{proof}
Adopting the same notation as in Lemma \ref{lemma11} and using, again,
integration by parts and the evolution equation we have the formula 
\begin{align*}
& \frac{d\beta}{dt} = -2\int_{0}^{2\pi }\frac{\partial ^{4}k}{\partial \theta ^{4}}\left( \frac{%
\partial ^{2}F}{\partial \theta ^{2}}k^{2}\frac{\partial ^{2}k}{\partial
\theta ^{2}}+2\frac{\partial F}{\partial \theta }k\frac{\partial k}{\partial
\theta }\frac{\partial ^{2}k}{\partial \theta ^{2}}+\frac{\partial F}{%
\partial \theta }k^{2}\frac{\partial ^{3}k}{\partial \theta ^{3}}+2\frac{%
\partial F}{\partial \theta }k\frac{\partial k}{\partial \theta }\frac{%
\partial ^{2}k}{\partial \theta ^{2}}+\right. \\
& \left. +2F\left( \frac{\partial k}{\partial \theta }\right) ^{2}\frac{%
\partial ^{2}k}{\partial \theta ^{2}}+2Fk\left( \frac{\partial ^{2}k}{%
\partial \theta ^{2}}\right) ^{2}+2Fk\frac{\partial k}{\partial \theta }%
\frac{\partial ^{3}k}{\partial \theta ^{3}}+\frac{\partial F}{\partial
\theta }k^{2}\frac{\partial ^{3}k}{\partial \theta ^{3}}+2Fk\frac{\partial k%
}{\partial \theta }\frac{\partial ^{3}k}{\partial \theta ^{3}}+\right. \\
& \left. +Fk^{2}\frac{\partial ^{4}k}{\partial \theta ^{4}}+\frac{\partial
^{2}G}{\partial \theta ^{2}}k^{2}\frac{\partial k}{\partial \theta }+2k\frac{%
\partial G}{\partial \theta }\left( \frac{\partial k}{\partial \theta }%
\right) ^{2}+\frac{\partial G}{\partial \theta }k^{2}\frac{\partial ^{2}k}{%
\partial \theta ^{2}}+2\frac{\partial G}{\partial \theta }k\left( \frac{%
\partial k}{\partial \theta }\right) ^{2}+\right. \\
& \left. +2G\left( \frac{\partial k}{\partial \theta }\right) ^{3}+4Gk\frac{%
\partial k}{\partial \theta }\frac{\partial ^{2}k}{\partial \theta ^{2}}+%
\frac{\partial G}{\partial \theta }k^{2}\frac{\partial ^{2}k}{\partial
\theta ^{2}}+2Gk\frac{\partial k}{\partial \theta }\frac{\partial ^{2}k}{%
\partial \theta ^{2}}+Gk^{2}\frac{\partial ^{3}k}{\partial \theta ^{3}}%
+\right. \\
& \left. +6k\left( \frac{\partial k}{\partial \theta }\right) ^{2}+3k^{2}%
\frac{\partial ^{2}k}{\partial \theta ^{2}}\right) d\theta
\end{align*}%
By the same trick used in Lemma \ref{lemma11} we can transform away the
fourth derivative. Using bounds for $k$, $\displaystyle\frac{\partial k}{%
\partial \theta }$, $\displaystyle\int_{0}^{2\pi }\left( \frac{\partial ^{2}k%
}{\partial \theta ^{2}}\right) ^{4}d\theta $ and $\displaystyle%
\int_{0}^{2\pi }\left( \frac{\partial ^{2}k}{\partial \theta ^{2}}\right)
^{2}d\theta $ and the fact that $k$ is bounded away from zero in $%
S^{1}\times \lbrack 0,T)$ by $k_{\mathrm{MIN}}(0)$ we have 
\begin{equation*}
\frac{d\beta }{dt}\leq C_{1}\beta +C_{2}
\end{equation*}%
where $C_{1}$ and $C_{2}$ are constants that don't depend on $t$. Now,
Gronwall's inequality gives that $\beta $ is bounded in $[0,T)$.
\end{proof}

\begin{prop}
\label{prop6} If $k$ is bounded in $S^1\times [0,T)$ then $k^{\prime \prime
} $ is also bounded in $S^1\times [0,T)$.
\end{prop}

\begin{proof}
We use the Poincar\'{e} inequality: if $u\in C^{1}([0,2\pi ])$ then%
\begin{equation*}
\max_{\lbrack 0,2\pi ]}|u|\leq \frac{1}{2\pi }\int_{0}^{2\pi
}u+\int_{0}^{2\pi }|u^{\prime }|
\end{equation*}%
Fix $t\in \lbrack 0,T)$. Then%
\begin{equation*}
\max_{\lbrack 0,2\pi ]}k^{\prime \prime 2}\leq \frac{1}{2\pi }\int_{0}^{2\pi
}k^{\prime \prime 2}d\theta +2\int_{0}^{2\pi }|k^{\prime \prime }(\theta
,t)k^{\prime \prime \prime }(\theta ,t)|\ d\theta
\end{equation*}%
Using Schwarz's inequality on the last integral and the previous lemmas we
have that $\displaystyle\max_{[0,2\pi ]}k^{\prime \prime 2}$ is bounded by a
constant that doesn't depend on $t$. This completes the proof.
\end{proof}

\begin{prop}
\label{prop7} If $k$ is bounded then all the spatial derivatives of $k$ are
also bounded.
\end{prop}

\begin{proof}
We have already proved that the two first derivatives are bounded. Using
this, we will prove that $\displaystyle\frac{\partial ^{3}k}{\partial \theta
^{3}}$ is bounded. Consider the function $u:S^{1}\times \lbrack
0,T)\rightarrow \mathbb{R}$ given by $u=e^{ct}\displaystyle\frac{\partial
^{3}k}{\partial \theta ^{3}}$. This function is a solution to a linear
parabolic second order equation of type 
\begin{equation*}
\frac{\partial u}{\partial t}=Fk^{2}\frac{\partial ^{2}u}{\partial \theta
^{2}}+P\frac{\partial u}{\partial \theta }+(c+Q)u+R,
\end{equation*}%
where $P,Q$ and $R$ are polynomials whose variables are the functions $k$, $%
F $, $G$ and its derivatives. Since the derivatives of $k$ only appear until
the second order in the terms $Q$ and $R$ one can use the maximum principle
for a suitable $c$ to show that $u$ is bounded. It follows that $%
\displaystyle\frac{\partial ^{3}k}{\partial \theta ^{3}}$ is bounded for
finite time. For the higher derivatives the argument is essentially the same.
\end{proof}

\begin{coro}
\label{coro2} If $k$ is bounded then its time derivatives of all orders are
also bounded.
\end{coro}

\begin{proof}
All the time derivatives depends polynomially on the spatial derivatives.
Then, uniform bounds on the spatial derivatives yields uniform bounds to the
time derivatives.
\end{proof}

\begin{teo}
\label{teo6} The solution to the minkowskian curvature evolution PDE
continues until the area converges to zero.
\end{teo}

\begin{proof}
We just proved that if $\displaystyle\lim_{t\rightarrow T}A(t)>0$ then $k$
and all of its derivatives remain bounded. By the Arzela's theorem $k$ has a
limit as $t$ goes to $T$ which is $C^{\infty }$. This shows that as long as
the area remains bounded away from zero we can extend the solution, and then
the solution exists until the area goes to $0$.
\end{proof}

\end{document}